%% file: plain-arxiv.tex
\documentclass[letterpaper,11pt]{article}

\input{preamble.tex}

\usepackage{authblk}
\usepackage[margin=1in]{geometry}

\newtheorem{theorem}{Theorem}
\newtheorem{example}{Example}
\newtheorem{lemma}{Lemma}
\newtheorem{remark}{Remark}

\begin{document}

\title{Correction to: A Lagrangian dual method for two-stage robust optimization with binary uncertainties}

\author[1]{Henri Lefebvre}
\author[2]{Anirudh Subramanyam}
\affil[1]{Trier University, Trier, Germany}
\affil[2]{The Pennsylvania State University, University Park, PA, USA}

\maketitle

\begin{abstract}
    We provide a correction to the sufficient conditions under which closed-form expressions for the optimal Lagrange multiplier are provided in \citet{subramanyam2022lagrangian}. We first present a simple counterexample where the original conditions are insufficient, highlight where the original proof fails, and then provide modified conditions along with a correct proof of their validity. Finally, although the original paper discusses modifications to their method for problems that may not satisfy any sufficient conditions, we substantiate that discussion along two directions. We first show that computing an optimal Lagrange multiplier can still be done in polynomial time. We then provide complete and correct versions of the corresponding Benders and column-and-constraint generation algorithms in which the original method is used. We also discuss the implications of our findings on computational performance.
\end{abstract}

\input{content.tex}

\section*{Acknowledgements}
H. Lefebvre acknowledges support by the German Bundesministerium f\"{u}r Bildung und Forschung within the project ``RODES'' (F\"{o}rderkennzeichen 05M22UTB). A. Subramanyam acknowledges support by the U.S. National Science Foundation under grant DMS-2229408.

\bibliographystyle{plainnat}
\bibliography{references.bib}

\end{document}

%% file: preamble.tex
\usepackage{amsmath,amssymb,amsfonts,mathtools,bm}
\usepackage{epsfig,graphicx,xcolor,natbib,pgfplots,url}
\usepackage{enumitem}
\usepackage{algorithm,algcompatible}
\usepackage{tabularx,booktabs,multirow}
\newcolumntype{C}{>{\centering\arraybackslash}X}
\newcolumntype{R}{>{\raggedleft\arraybackslash}X}
\newcolumntype{L}{>{\raggedright\arraybackslash}X}
\algnewcommand{\IIf}[1]{\State\algorithmicif\ #1\ \algorithmicthen}
\algnewcommand{\EndIIf}{\unskip\ \algorithmicend\ \algorithmicif}
\algnewcommand{\IElse}{\unskip\ \algorithmicelse \unskip\ }
\makeatletter
\newcommand{\multiline}[1]{%
    \begin{tabularx}{\dimexpr\linewidth-\ALG@thistlm}[t]{@{}X@{}}
        #1
    \end{tabularx}
}
\makeatother

\makeatletter \newcommand{\pgfplotsdrawaxis}{\pgfplots@draw@axis} \makeatother

\pgfplotsset{
    compat=1.18,
    every tick/.append style={color=black},
    after end axis/.append code={
            \pgfplotsset{
                axis line style=opaque,
                ticklabel style=opaque,
                tick style=opaque,
                grid=none
            }
            \pgfplotsdrawaxis
        },
} %

\newcommand{\one}{\bm{\mathrm{e}}}

\allowdisplaybreaks

\usepackage{amsthm}

%% file: content.tex
\section{Introduction}

In~\citet{subramanyam2022lagrangian}, the author considers two-stage robust optimization problems with binary-valued uncertain data and proposes a new method to construct worst-case parameter realizations in such problems. The method is embedded in traditional Benders decomposition and column-and-constraint generation generation algorithms to obtain a new exact solution method for two-stage robust optimization problems that is computationally superior compared to traditional methods. It relies on the development of a suitable Langrangian dual with a single Lagrange multiplier. Although not central to the method, the paper also provides practically computable closed-form expressions for the optimal Lagrange multiplier in problem where certain sufficient conditions are satisfied. We henceforth refer to~\citet{subramanyam2022lagrangian} and the developments therein as ``original'' to distinguish our own contributions.

In this paper, we show that the original sufficient conditions for optimality of the Lagrange multiplier, which we review in Section~\ref{sec:background}, are incorrect. We illustrate this using a simple single-variable problem in Section~\ref{sec:counterexample}. The majority of our paper then focuses on the theoretical and computational implications of this observation. Theoretically, we show that the same closed-form expression for the optimal multiplier continues to remain applicable when certain stronger conditions are satisfied; see Section~\ref{sec:corrected}.
These conditions may not be satisfied in practice; therefore, we discuss key implications of their absence in Section~\ref{sec:absence}. We first show in Section~\ref{sec:complexity} that an alternate closed-form expression for the optimal multiplier can still be computed using a polynomial time function of the input data under more general conditions. We supplement this theoretical analysis with practical algorithmic strategies in Section~\ref{sec:algo_modifications}, since our findings open the possibility that the originally reported computational results are no longer exact. Hence, we propose strategies that supplement the original method with a small number of easily implementable correction steps (termed ``modifications''), which allow us to recover an exact method for general two-stage problems that need not satisfy any extra conditions. We report numerical results in Section~\ref{sec:computational} to demonstrate that the computational overhead of the proposed modifications are marginal: they do not come at the expense of sacrificing the computational superiority of the original method. 

\section{Background}\label{sec:background}
To keep our presentation succinct, we adopt all notation and assumptions from the original paper, and consider the problem formulation denoted as $\mathcal{P}$, shown below.
\begin{equation}\label{eq:two_stage_ro_general}\tag{$\mathcal{P}$}
    \begin{aligned}
         & \inf_{\bm{x} \in \mathcal{X}} \sup_{\bm{\xi} \in \Xi} \, \mathcal{Q}(\bm{x}, \bm{\xi}), \\
         &                                                                                         %
        \mathcal{Q}(\bm{x}, \bm{\xi}) =
        \left[
            \begin{aligned}
                \mathop{\text{minimize}}_{\bm{y} \in \mathcal{Y}} \;\; & \bm{c}(\bm{\xi})^\top \bm{x}  + \bm{d}(\bm{\xi})^\top \bm{y} \\
                \text{subject to} \;                                   & \bm{T}\bm{x} + \bm{W} \bm{y} \geq \bm{h}(\bm{\xi})
            \end{aligned}
            \right].
    \end{aligned}
\end{equation}
Here, $\bm{x}$, $\bm{y}$ and $\bm{\xi}$ are the first- and second-stage decision variables, and the vector of uncertain parameters, respectively.
The feasible decision sets $\mathcal{X} \subseteq \mathbb{R}^{n_1}$ and $\mathcal{Y} \subseteq \mathbb{R}^{n_2}$, and the uncertainty set $\Xi \subseteq \{0, 1\}^{n_p}$ are assumed non-empty and mixed-integer linear programming (MILP) representable;
$\mathcal{X}$ is compact;
and $\mathcal{Q} : \mathbb{R}^{n_1} \times \{0, 1\}^{n_p} \to \mathbb{R} \cup \{-\infty, +\infty\}$ is the second-stage value function, which also includes the first-stage costs $\bm{c}(\bm{\xi})^\top \bm{x}$ for simplicity.
The optimal value of a minimization (maximization) problem is defined to be $-\infty$ ($+\infty$) if it is unbounded and $+\infty$ ($-\infty$) if it is infeasible.

The central idea of the method in \citet{subramanyam2022lagrangian} is the development of the following Lagrangian dual with scalar-valued multiplier $\lambda \in \mathbb{R}_{+}$.
\begin{align}
    \mathcal{L}(\bm{x}, \bm{\xi}, \lambda)
     & =\left[
        \begin{aligned}
            \mathop{\text{minimize}}_{\bm{y} \in \mathcal{Y}, \bm{z} \in \mathbb{R}^{n_p}_{+}} \;\; &
            \bm{c}(\bm{\xi})^\top \bm{x}  + \bm{d}(\bm{\xi})^\top \bm{y} + \lambda \phi(\bm{z}, \bm{\xi})                                                                      \\
            \text{subject to} \;\;                                                                  & \bm{T}\bm{x} + \bm{W} \bm{y} \geq \bm{h}(\bm{z}), \;\; \bm{z} \leq \one.
        \end{aligned}
        \right]
    \label{eq:lagrangian_general}                                        \\
    \phi(\bm{z}, \bm{\xi})
     & = \one^\top \bm{z} + \one^\top \bm{\xi} - 2 \bm{z}^\top \bm{\xi}.
    \label{eq:penalty_general}
\end{align}
Here, $\phi$ plays the role of a penalty function, as the uncertain parameters $\bm{\xi}$ appear only in the objective function of $\mathcal{L}$.
It is shown in~\citet[Theorems~1 and~2]{subramanyam2022lagrangian} that the Lagrangian dual has the following attractive properties.
\begin{align}
    \mathcal{Q}(\bm{x}, \bm{\xi})                                                       & = \sup_{\lambda \in \mathbb{R}_{+}} \mathcal{L}(\bm{x}, \bm{\xi}, \lambda)
    \; \text{ for all } \bm{x} \in \mathcal{X}, \bm{\xi} \in \Xi.
    \label{eq:strong_duality_general}
    \\
    \inf_{\bm{x} \in \mathcal{X}} \sup_{\bm{\xi} \in \Xi} \mathcal{Q}(\bm{x}, \bm{\xi}) & = \sup_{\lambda \in \mathbb{R}_{+}} \inf_{\bm{x} \in \mathcal{X}} \sup_{\bm{\xi} \in \Xi} \mathcal{L}(\bm{x}, \bm{\xi}, \lambda)
    \label{eq:strong_duality_general_two_stage}
\end{align}
These equations establish strong duality even if the second-stage feasible region $\mathcal Y$ may not be convex. Although not explicitly stated in the original paper, an
immediate consequence of these properties is the following equation to which we shall refer in this paper. This equation is obtained by simply taking the supremum of both sides of~\eqref{eq:strong_duality_general} over $\bm{\xi} \in \Xi$.
\begin{equation}
    \sup_{\bm{\xi} \in \Xi} \mathcal{Q}(\bm{x}, \bm{\xi}) = \sup_{\lambda \in \mathbb{R}_{+}} \sup_{\bm{\xi} \in \Xi} \mathcal{L}(\bm{x}, \bm{\xi}, \lambda)
    \; \text{ for all } \bm{x} \in \mathcal{X}.
    \label{eq:strong_duality_general_fixed_x}
\end{equation}

In \citet[Theorem~4]{subramanyam2022lagrangian}, it is shown that one can compute a closed-form expression for the optimal Lagrange multiplier; that is, one that maximizes the right-hand side of~\eqref{eq:strong_duality_general_fixed_x} if ``for every $\bm{x} \in \mathcal{X}$, either there exists $\bm{\xi} \in \Xi$ such that $\mathcal{Q}(\bm{x}, \bm{\xi}) = +\infty$ or $\mathcal{Q}(\bm{x}, \bm{\xi}) < +\infty$ for all $\bm{\xi} \in \{0, 1\}^{n_p}$''.
Whenever this condition holds, the paper claims that for any feasible first-stage decision $\bm{x} \in \mathcal{X}$;
that is,
for which $\sup \big\{\mathcal{Q}(\bm{x}, \bm{\xi}) : \bm{\xi} \in \Xi \big\} < +\infty$,
\begin{equation}\label{eq:false_claim}
    u(\bm{x}) - \ell(\bm{x})
    \in
    \mathop{\arg\max}_{\lambda \in \mathbb{R}_{+}} \left\{
    \max_{\bm{\xi} \in \Xi} \mathcal{L}(\bm{x}, \bm{\xi}, \lambda)
    \right\},
\end{equation}
where $u(\bm{x})$ is any finite upper bound on
$\sup \big\{\mathcal{Q}(\bm{x}, \bm{\xi}) : \bm{\xi} \in \Xi \big\}$
and  $\ell(\bm{x})$ is any finite lower bound on
$\inf \big\{\bm{c}(\bm{\xi})^\top \bm{x} + \bm{d}(\bm{\xi})^\top \bm{y} : \bm{\xi} \in \Xi, \bm{y} \in \mathcal{Y} \big\}$.
In other words, $u(\bm{x}) - \ell(\bm{x})$ is an optimal Lagrange multiplier.

\section{Counterexample}\label{sec:counterexample}

Consider the problem with $\mathcal X = \{ 0 \}$, $\Xi = \{ 0, 1 \}$, $\mathcal Y = \{0,1\}$, and whose second-stage value function is given by
\begin{equation*}
    \mathcal Q(x,\xi) = \left[
        \begin{array}{ll}
            \underset{y\in\mathcal Y}{\text{minimize}} \  & -y                   \\
            \text{subject to } \                          & y \le \dfrac32 - \xi
        \end{array}
        \right].
\end{equation*}
For simplicity, this problem has a unique feasible first-stage decision, so that the second-stage value function depends only on the uncertain parameter $\xi$. The counterexample can, however, be easily extended to more complex first-stage decision spaces.

It can be readily verified by enumeration that $\mathcal Q(x,\xi) < +\infty$ for all $x \in \mathcal X$ and all $\xi \in \Xi$. Hence, the conditions of \citet[Theorem~4]{subramanyam2022lagrangian} are satisfied. We show, however, that the claimed result is wrong. %
To that end, observe the following obtained by simply enumerating all points in $\Xi$ and $\mathcal{Y}$:
\begin{subequations}
    \begin{align*}
        \sup_{\xi\in\Xi} \mathcal Q(x,\xi)
         & = \max\left\{ \min_{y \in \{0, 1\}, y \le \frac32} -y, \;\; \min_{y\in \{0, 1\}, y \le \frac12} -y \right\}
        = \max\left\{ -1,  0 \right\}
        = 0,                                                                                                           \\
        \inf_{ y \in  \{0, 1\}} -y
         & = \min\{ 0, -1 \}
        = -1.
    \end{align*}
\end{subequations}
According to \citet[Theorem~4]{subramanyam2022lagrangian}, one can choose $u(x) = 0$ and $\ell(x) = -1$ to ensure that $u(x) - \ell(x) = 1$ is an optimal Lagrange multiplier, namely that $1 \in \mathop{\arg\max}_{\lambda \in \mathbb{R}_{+}} \left\{
    \max_{\xi \in \Xi} \mathcal{L}(x, \xi, \lambda)
    \right\}$.
We now proceed to show that this is false %
by explicitly calculating $\mathcal L(x,\xi,\lambda)$ by brute force enumeration. In deriving the following, we only use the fact that $\lambda \geq 0$.
\begin{align*}
    \max_{\xi\in\Xi} \mathcal L(x,\xi,\lambda) & =
    \max_{\xi \in \{0, 1\}} \left\{\min_{\substack{(y, z) \in \{0, 1\} \times [0, 1],                                                           \\ y + z \le \frac32}}-y + \lambda \left(z + \xi - 2z\xi \right) \right\} \\
                                               & = \max \left\{
    \min_{\substack{(y, z) \in \{0, 1\} \times [0, 1],                                                                                          \\ y + z \le \frac32}} -y + \lambda z ,
    \min_{\substack{(y, z) \in \{0, 1\} \times [0, 1],                                                                                          \\ y + z \le \frac32}} -y + \lambda - \lambda z
    \right\}                                                                                                                                    \\
                                               & = \max \Bigg\{
    \min\left\{\min_{z \in [0, 1], z \le \frac32} \lambda z, \min_{z \in [0, 1], z \le \frac12} -1 + \lambda z \right\},                        \\
                                               & \phantom{= \max \Bigg\}\;}
    \min\left\{\min_{z \in [0, 1], z \le \frac32} \lambda - \lambda z, \min_{z \in [0, 1], z \le \frac12} -1 + \lambda - \lambda z \right\}
    \Bigg\}                                                                                                                                     \\
                                               & =\max\left\{\min \left\{0, -1 \right\}, \min \left\{0,  \frac{\lambda}{2} - 1 \right\}\right\} \\
                                               & = \max\left\{-1, \min \left\{0, \frac{\lambda}{2} - 1 \right\} \right\}.
\end{align*}
A plot of the function is shown in Figure~\ref{fig:counterexample_function}, which indicates that the function is not maximized at $\lambda = u(x) - \ell(x) = 1$.
In other words, equation~\eqref{eq:false_claim} is false.

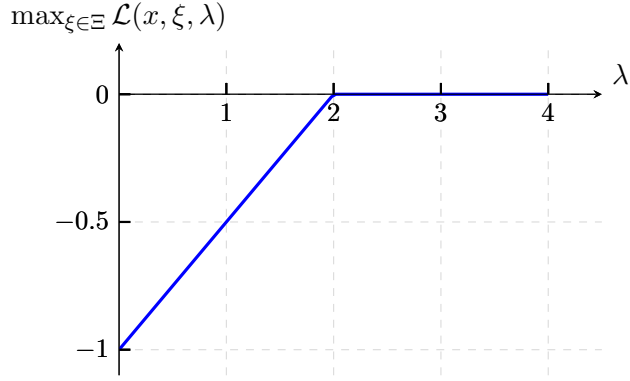
\begin{figure}
    \centering
    \begin{tikzpicture}
        \begin{axis}[
                width=8cm,                      %
                height=6cm,                      %
                xlabel={$\lambda$},              %
                ylabel={$\max_{\xi\in\Xi} \mathcal{L}(x,\xi,\lambda)$}, %
                axis lines=middle,               %
                tick align=inside,               %
                tick style={thick},              %
                xlabel style={font=\small},      %
                ylabel style={font=\small},      %
                tick label style={font=\small},  %
                grid=major,                      %
                grid style={dashed, gray!30},    %
                ymin=-1.1, ymax=0.2,             %
                xmin=0, xmax=4.5,                   %
                every axis x label/.style={at={(current axis.right of origin)},anchor=south west},
                every axis y label/.style={at={(current axis.above origin)},anchor=south},
                xtick distance=1,
                extra y ticks= {0},
            ]

            \addplot [
                domain=0:4,
                samples=100,
                very thick,
                blue
            ]
            { max( -1, min(0, 0.5*x - 1) ) };

        \end{axis}
    \end{tikzpicture}
    \caption{Plot of $\max_{\xi\in\Xi} \mathcal{L}(x,\xi,\lambda)$ versus $\lambda$ for the counterexample. \label{fig:counterexample_function}}
\end{figure}

\begin{remark}
		The counterexample can be extended to show that the optimality gap can be arbitrarily large. Suppose that we scale the objective function of the counterexample by some fixed $\gamma > 0$. Then, it can be readily verified that the second-stage value function becomes $\gamma \mathcal Q(x,\xi)$ and that \citet[Theorem~4]{subramanyam2022lagrangian} prescribes $\gamma \cdot (u(x) - \ell(x)) = \gamma$ to be an optimal multiplier for the scaled problem. However, $\max_{\xi\in\Xi} \mathcal L(x,\xi,\gamma)$ is equal to $-\gamma/2$, whereas $\max_{\xi\in\Xi} \gamma\mathcal Q(x,\xi)$ is equal to $0$. In other words, one obtains an additive optimality gap of $\gamma / 2$, independent of the choice of~$\gamma$. This shows that the Lagrangian dual, when evaluated at the originally claimed optimal value of $\lambda$, can lead to arbitrarily poor relaxations of the second-stage value function. 
    \end{remark}

\section{Correct Sufficient Conditions}\label{sec:corrected}

The proof presented in~\citet[Theorem~4]{subramanyam2022lagrangian} fails when it is claimed that ``problem~\eqref{eq:lagrangian_general} always has an optimal solution $(\hat{\bm{y}}, \hat{\bm{z}})$ such that $\hat{\bm{z}} \in \{0, 1\}^{n_p}$, for any $\bm{\xi} \in \Xi$ and $\lambda \in \mathbb{R}_{+}$''.
It is worth noting, however, that all other steps of the proof remain valid whenever this claim is true. %
Below, we provide complete and correct conditions under which this claim is true and hence, the conclusion of~\citet[Theorem~4]{subramanyam2022lagrangian} remains valid.
Additionally, we highlight that although subsequent results in the original paper (e.g., Theorem~5, Algorithm~3,  Algorithm~8) invoke \citet[Theorem~4]{subramanyam2022lagrangian}, they do not require any modifications themselves. Indeed, their validity is unaffected as long as the correct sufficient conditions are satisifed.

For simplicity of our ensuing exposition, we disaggregate the second-stage decisions in~\ref{eq:two_stage_ro_general} into their continuous and discrete components.
Let $\bm y = (\bm y_\mathrm{c}, \bm y_\mathrm{d})$, where $\bm y_\mathrm{c}$ and $\bm y_\mathrm{d}$ are the vector of continuous and discrete second-stage variables, respectively. We extend this notation to matrices and vectors which are multiplied by $\bm y$ (for example, $\bm W\bm y = \bm W_\mathrm{c}\bm y_\mathrm{c} + \bm W_\mathrm{d}\bm y_\mathrm{d}$), so that $\mathcal{Q}(\bm{x}, \bm{\xi})$ can also be equivalently written as:
\begin{equation*}%
    \begin{aligned}
        \mathcal{Q}(\bm{x}, \bm{\xi}) =
        \left[
            \begin{aligned}
                \mathop{\text{minimize}}_{\bm{y}} \;\; & \bm{c}(\bm{\xi})^\top \bm{x}  + \bm{d}_\mathrm{c}(\bm{\xi})^\top \bm{y}_\mathrm{c}  + \bm{d}_\mathrm{d}(\bm{\xi})^\top \bm{y}_\mathrm{d} \\
                \text{subject to} \;                   & \bm{T}\bm{x} + \bm{W}_\mathrm{c} \bm{y}_\mathrm{c} + \bm{W}_\mathrm{d} \bm{y}_\mathrm{d} \geq \bm{h}(\bm{\xi})                           \\
                                                       & \bm{y} = (\bm{y}_\mathrm{c}, \bm{y}_\mathrm{d}) \in \mathcal Y \coloneqq \mathbb{R}^{nc_2}_{+} \times \mathcal{Y}_\mathrm{d}
            \end{aligned}
            \right],
    \end{aligned}
\end{equation*}
where $\mathcal{Y}_\mathrm{d} \subseteq \mathbb Z^{nd_2}$, $nc_2, nd_2 \in \mathbb Z \cap [0, n_2]$ and $nc_2 + nd_2 = n_2$.
This representation is general enough to allow the second-stage decisions to be purely continuous ($nc_2 = n_2$), purely integer ($nd_2 = n_2$) or mixed-integer.

We also define $\mathcal Q(\bm x, \bm \xi ; \bm y_\mathrm{d})$ and $\mathcal L(\bm x, \bm \xi, \lambda; \bm y_\mathrm{d})$ as restrictions of $\mathcal Q(\bm x, \bm \xi)$ and $\mathcal L(\bm x, \bm \xi, \lambda)$, respectively, where the values of the discrete second-stage decisions are fixed to $\bm y_\mathrm{d}$, as shown below, where $\bm y_\mathrm{d}$ is omitted below the minimize sign.
\begin{align} \label{eq:Q_fixed_yd}
    \mathcal{Q}(\bm{x}, \bm{\xi}; \bm y_\mathrm{d})
     & =
    \left[
        \begin{aligned}
            \mathop{\text{minimize}}_{\bm{y}_\mathrm{c} \in \mathbb{R}^{nc_2}_{+}}\;\; &
            \bm{c}(\bm{\xi})^\top \bm{x}  + \bm{d}_\mathrm{c}(\bm{\xi})^\top \bm{y}_\mathrm{c} + \bm{d}_\mathrm{d} (\bm{\xi})^\top \bm y_\mathrm{d} \\
            \text{subject to} \;\;                                                     &
            \bm{T}\bm{x} + \bm{W}_\mathrm{c} \bm{y}_\mathrm{c} + \bm W_\mathrm{d} \bm y_\mathrm{d} \geq \bm{h}(\bm{\xi})
        \end{aligned}
    \right], \\ \label{eq:L_fixed_yd}
    \mathcal{L}(\bm{x}, \bm{\xi}, \lambda; \bm y_\mathrm{d})
     & =
    \left[
        \begin{aligned}
            \mathop{\text{minimize}}_{\bm{y}_\mathrm{c} \in \mathbb{R}^{nc_2}_{+},  \bm{z} \in \mathbb{R}^{n_p}_{+}}\;\; &
            \bm{c}(\bm{\xi})^\top \bm{x}  + \bm{d}_\mathrm{c}(\bm{\xi})^\top \bm{y}_\mathrm{c} + \bm{d}_\mathrm{d}(\bm\xi)^\top\bm y_\mathrm{d} + \lambda \phi(\bm{z}, \bm{\xi}) \\
            \text{subject to} \;\;\;\;\;                                                                                 &
            \bm{T}\bm{x} + \bm{W}_\mathrm{c} \bm{y}_\mathrm{c} + \bm W_\mathrm{d}\bm y_\mathrm{d} \geq \bm{h}(\bm{z}), \;\; \bm{z} \leq \one.
        \end{aligned}
        \right].
\end{align}

The next theorem provides a correct set of sufficient conditions under which the original closed-form expression, namely $u(\bm x) - \ell(\bm x)$, remains an optimal value of the Lagrange multiplier.
\begin{theorem}%
    \label{theorem:optimal_multiplier_corrected}
    Suppose that the following conditions are satisfied in problem~\ref{eq:two_stage_ro_general}.
    \begin{enumerate}
        \item $\mathcal{X} \subseteq \mathbb{Z}^{n_1}$.
        \item $\bm{T} \in \mathbb{Z}^{m \times n_1}$, %
              $\bm{W} \in \mathbb{Z}^{m \times {n}_2}$, $\bm{h}(\bm{\xi}) = \bm{h}_0 + \bm{H} \bm{\xi}$, where $\bm{h}_0 \in \mathbb{Z}^m$, $\bm{H} \in \mathbb{Z}^{m \times n_p}$.
        \item $[\bm W_\mathrm{c}\; -\bm H] \in \mathbb{Z}^{m \times (nc_2 + n_p)}$ is a totally unimodular matrix.
    \end{enumerate}
    Then, for any feasible first-stage decision $\bm{x} \in \mathcal{X}$;
    that is,
    for which $\sup \big\{\mathcal{Q}(\bm{x}, \bm{\xi}) : \bm{\xi} \in \Xi \big\} < +\infty$,
    we have that
    \begin{equation*}
        u(\bm{x}) - \ell(\bm{x})
        \in
        \mathop{\arg\max}_{\lambda \in \mathbb{R}_{+}} \left\{
        \max_{\bm{\xi} \in \Xi} \mathcal{L}(\bm{x}, \bm{\xi}, \lambda)
        \right\},
    \end{equation*}
    where $u(\bm{x})$ is any finite upper bound on
    $\sup \big\{\mathcal{Q}(\bm{x}, \bm{\xi}) : \bm{\xi} \in \Xi \big\}$
    and  $\ell(\bm{x})$ is any finite lower bound on
    $\inf \big\{\bm{c}(\bm{\xi})^\top \bm{x} + \bm{d}(\bm{\xi})^\top \bm{y} : \bm{\xi} \in \Xi, \bm{y} \in \mathcal{Y} \big\}$.
\end{theorem}
\begin{proof}
    Suppose that $\bm{x} \in \mathcal{X}$ is any feasible first-stage decision in \ref{eq:two_stage_ro_general}.
    We shall show that under the conditions stated in the Theorem, problem~\eqref{eq:lagrangian_general} always has an optimal solution $(\hat{\bm{y}}, \hat{\bm{z}})$ such that $\hat{\bm{z}} \in \{0, 1\}^{n_p}$, for any $\bm{\xi} \in \Xi$ and $\lambda \in \mathbb{R}_{+}$. The rest of the argument follows from the proof of~\citet[Theorem~4]{subramanyam2022lagrangian} and remains unchanged.

    Using the definition of $\mathcal{L}(\bm{x}, \bm{\xi}, \lambda ; \bm{y}_\mathrm{d})$ in equation~\eqref{eq:L_fixed_yd},  problem~\eqref{eq:lagrangian_general} can be equivalently written as the following nested optimization problem:
    \begin{equation*}
        \mathcal{L}(\bm{x}, \bm{\xi}, \lambda)
        =\mathop{\text{minimize}}_{\bm{y}_\mathrm{d} \in \mathcal{Y}_\mathrm{d}} \mathcal{L}(\bm{x}, \bm{\xi}, \lambda ; \bm{y}_\mathrm{d})
    \end{equation*}
    Under the stated conditions, it can be readily verified the constraint matrix %
    defining the feasible region of $\mathcal{L}(\bm{x}, \bm{\xi}, \lambda ; \bm{y}_\mathrm{d})$ %
    is totally unimodular.
    Moreover, the right-hand side coefficients,
    $\bm{h}_0 - \bm{T}\bm{x} - \bm{W}_\mathrm{d} \bm{y}_\mathrm{d}$,
    are integer-valued for any $\bm{x} \in \mathcal{X} \subseteq \mathbb{Z}^{n_1}$ and $\bm{y}_\mathrm{d} \in \mathcal{Y}_\mathrm{d} \subseteq \mathbb{Z}^{nd_2}$.
    Therefore, the polyhedron defining the feasible region of $\mathcal{L}(\bm{x}, \bm{\xi}, \lambda ; \bm{y}_\mathrm{d})$ has integer vertices,
    as does its optimal solution.
    Hence, any optimal solution $(\hat{\bm{y}}, \hat{\bm{z}})$ of $\mathcal{L}(\bm{x}, \bm{\xi}, \lambda)$ must satisfy $\hat{\bm{z}} \in \{0, 1\}^{n_p}$.%
\end{proof}

The conditions in Theorem~\ref{theorem:optimal_multiplier_corrected} are much stronger than those presented in the original paper.
	This is not surprising given the simplicity of the counterexample presented in Section~\ref{sec:counterexample}, which is a pure integer problem with a single variable and a single constraint. %
    In Section~\ref{sec:complexity}, we show that a much more general class of problems admits (polynomial time computable) closed-form expressions for an optimal Lagrange multiplier. However, this result is mostly of theoretical value, since it does not offer a practically computable expression for the multiplier.
    Therefore, in Section~\ref{sec:algo_modifications}, we %
    present practical algorithmic schemes to compute an optimal multiplier that remain valid and rigorous even in problem instances that may not satisfy any of the sufficient conditions presented either in this section or those in Section~\ref{sec:complexity}.

Although the conditions in Theorem~\ref{theorem:optimal_multiplier_corrected} may appear restrictive, there is a broad family of problems where the conditions are naturally satisfied.
\begin{example}[Interdicted Combinatorial Problems]
    Suppose that the second-stage problem is combinatorial, $\mathcal Y= \{0,1\}^{n_2}$, and the second-stage decisions represent resources that are being interdicted depending on some random realization of the uncertain parameters, as shown below.
    Such structures are common in network interdiction problems; e.g., see~\citet{Lefebvre2024,Lefebvre2024b}.
    \begin{equation*}
        \mathcal{Q}(\bm{x}, \bm{\xi}) =
        \left[
            \begin{aligned}
                \mathop{\text{minimize}}_{\bm{y}\in \mathcal Y} \;\; & \bm{c}(\bm{\xi})^\top \bm{x}  + \bm{d}(\bm{\xi})^\top \bm{y} \\
                \text{subject to} \;                                 & \bm{T}\bm{x} + \bm{W} \bm{y} \geq \bm{h},                    \\
                                                                     & \bm 0 \le \bm y \le \bm e - \bm \xi
            \end{aligned}
            \right].
    \end{equation*}
    Suppose also that $\mathcal X\subseteq\mathbb Z^{n_1}$ and that all matrices are integer. Then, it can be verified that the assumptions of Theorem~\ref{theorem:optimal_multiplier_corrected} are satisfied (with $nc_2 = 0$). One can then compute a closed-form expression for the optimal multiplier using expressions for $u(\bm x)$ and $\ell(\bm x)$ provided in \citet[Theorem~5]{subramanyam2022lagrangian}.
\end{example}

Our proof argument allows us to also extend Theorem~\ref{theorem:optimal_multiplier_corrected} to problem~\ref{eq:two_stage_ro_indicator} (reproduced below from the original paper) without requiring significant changes.
\begin{equation}\label{eq:two_stage_ro_indicator}\tag{$\mathcal{P}_\mathcal{I}$}
    \begin{aligned}
         & \inf_{\bm{x} \in \mathcal{X}} \sup_{\bm{\xi} \in \Xi} \, \mathcal{Q}_\mathcal{I}(\bm{x}, \bm{\xi}), \\
         &                                                                                                     %
        \mathcal{Q}_\mathcal{I}(\bm{x}, \bm{\xi}) =
        \left[\begin{aligned}
                      \mathop{\text{minimize}}_{\bm{y} \in \mathcal{Y}} \;\; & \bm{c}(\bm{\xi})^\top \bm{x}  + \bm{d}(\bm{\xi})^\top \bm{y}                           \\
                      \text{subject to} \;                                   & \bm{g}(\bm{x}, \bm{y}) \geq \bm{0}                                                     \\
                                                                             & \xi_j = 0 \implies g_i(\bm{x}, \bm{y}) = 0, \;\; i \in \mathcal{I}_j^0, \; j \in [n_p] \\
                                                                             & \xi_j = 1 \implies g_i(\bm{x}, \bm{y}) = 0, \;\; i \in \mathcal{I}_j^1, \; j \in [n_p]
                  \end{aligned}\right].
    \end{aligned}
\end{equation}
Here, %
$\bm g(\bm x, \bm y) = \bm T \bm x + \bm W \bm y - \bm h_0$ is some vector-valued affine function and $\mathcal{I}_j^0, \mathcal{I}_j^1 \subseteq [m]$ are some index sets.
The indicator formulation \ref{eq:two_stage_ro_indicator} arises in applications where the uncertain presence or absence of a quantity switches on or off second-stage constraints. Its corresponding Lagrangian is given by:
\begin{align}
    \mathcal{L}_\mathcal{I}(\bm{x}, \bm{\xi}, \lambda)
     & =\left[
        \begin{aligned}
            \mathop{\text{minimize}}_{\bm{y} \in \mathcal{Y}} \;\; &
            \bm{c}(\bm{\xi})^\top \bm{x}  + \bm{d}(\bm{\xi})^\top \bm{y} + \lambda \phi_\mathcal{I}(\bm{x}, \bm{y}, \bm{\xi}) \\
            \text{subject to} \;\;                                 & \bm{g}(\bm{x}, \bm{y}) \geq \bm{0}.
        \end{aligned}
        \right]
    \label{eq:lagrangian_indicator}                                                                                                                                  \\
    \phi_\mathcal{I}(\bm{x}, \bm{y}, \bm{\xi})
     & = \sum_{j \in [n_p]} \sum_{i \in \mathcal{I}_j^1} \xi_j g_i(\bm{x}, \bm{y}) + \sum_{j \in [n_p]}\sum_{i \in \mathcal{I}_j^0} (1 - \xi_j) g_i(\bm{x}, \bm{y}),
    \label{eq:penalty_indicator}
\end{align}
where, as before, $\phi_\mathcal{I}$ is a penalty function that serves to move all uncertain parameters from the constraints to the objective function.

We now present the following analog of Theorem~\ref{theorem:optimal_multiplier_corrected} for problem~\ref{eq:two_stage_ro_indicator}.
Notably, we highlight that a similar result was not postulated in the original paper.
The proof is similar to that of Theorem~\ref{theorem:optimal_multiplier_corrected} and we omit it for the sake of brevity.
\begin{theorem}\label{theorem:optimal_multiplier_indicator}
    Suppose that the following conditions are satisfied in problem~\ref{eq:two_stage_ro_indicator}.
    \begin{enumerate}
        \item $\mathcal{X} \subseteq \mathbb{Z}^{n_1}$.
        \item $\bm{T} \in \mathbb{Z}^{m \times n_1}$, %
              $\bm{W} \in \mathbb{Z}^{m \times {n}_2}$,  $\bm{h}_0 \in \mathbb{Z}^m$.
        \item $\bm W_\mathrm{c} \in \mathbb{Z}^{m \times nc_2}$ is a totally unimodular matrix.
    \end{enumerate}
    Then, for any feasible first-stage decision $\bm{x} \in \mathcal{X}$;
    that is,
    for which $\sup \big\{\mathcal{Q}_\mathcal{I}(\bm{x}, \bm{\xi}) : \bm{\xi} \in \Xi \big\} < +\infty$,
    we have that
    \begin{equation*}
        u(\bm{x}) - \ell(\bm{x})
        \in
        \mathop{\arg\max}_{\lambda \in \mathbb{R}_{+}} \left\{
        \max_{\bm{\xi} \in \Xi} \mathcal{L}_\mathcal{I}(\bm{x}, \bm{\xi}, \lambda)
        \right\},
    \end{equation*}
    where $u(\bm{x})$ is any finite upper bound on
    $\sup \big\{\mathcal{Q}_\mathcal{I}(\bm{x}, \bm{\xi}) : \bm{\xi} \in \Xi \big\}$
    and  $\ell(\bm{x})$ is any finite lower bound on
    $\inf \big\{\bm{c}(\bm{\xi})^\top \bm{x} + \bm{d}(\bm{\xi})^\top \bm{y} : \bm{\xi} \in \Xi, \bm{y} \in \mathcal{Y} \big\}$.
\end{theorem}

\section{Absence of Sufficient Conditions}\label{sec:absence}

We now discuss key implications in the absence of sufficient conditions under which $u(\bm x) - \ell(\bm x)$ is an optimal value of the Lagrange multiplier.
Although the original paper includes a short discussion to that end, we present a more thorough and formal treatment in this section.
First, in Section~\ref{sec:complexity}, we show that computing an optimal multiplier can still be done in closed form and in time that is polynomial in the size of the input data under fairly general conditions. However, this positive complexity result is mostly of theoretical value, since the multiplier value is too large to be useful in practice.
Therefore, in Section~\ref{sec:algo_modifications}, we present exact algorithmic solutions that are broadly applicable to any two-stage robust optimization problems of the form \ref{eq:two_stage_ro_general} or \ref{eq:two_stage_ro_indicator}, irrespective of whether they satisfy the conditions that allow closed-form expressions of the Lagrange multiplier.

\subsection{Computational Complexity}
\label{sec:complexity}

To ease our presentation, we first paraphrase a result
from~\citet[Lemma~4]{buchheim2023bilevel} regarding solutions of linear
programs.

\begin{lemma}[\citet{buchheim2023bilevel}]
    \label{lemma:buchheim}
    Let $\lVert \bm X \rVert$ denote the maximum absolute value of any entry of a matrix or vector $\bm X$.
    Let
    $P = \{ \bm\omega \in \mathbb{R}_+^{n_\omega} : \bm A \bm \omega = \bm b  \}$ be a polyhedron with $\bm A \in \mathbb Z^{m_\omega \times n_\omega}$ and $\bm b \in\mathbb Z^{m_\omega}$. Then, any vertex $\bar{\bm \omega}$ of $P$ satisfies
    $\lVert \bar{\bm \omega} \rVert \le m_\omega ! \lVert \bm b \rVert \lVert \bm A \rVert^{m_\omega - 1} $.
\end{lemma}

We highlight that the bound in Lemma~\ref{lemma:buchheim} can be computed in time polynomial in the input data $(\bm A, \bm b)$. %
It allows us to prove our main complexity result, which we state next.

\begin{theorem}\label{theorem:polynomial_time_lambda}
    Suppose that the following conditions are satisfied in problem~\ref{eq:two_stage_ro_general}.
    \begin{enumerate}
        \item $\mathcal X\subseteq\mathbb Z^{n_1}\cap [{\bm x}^\ell, {\bm x}^u]$.
        \item %
              For all $\bm x\in\mathcal X$ and $\bm\xi\in\Xi$, we have $\{ \bm y\in\mathcal Y : \bm T\bm x + \bm W \bm y \ge \bm h(\bm\xi) \} \subseteq [{\bm y}^\ell, {\bm y}^u]$.
        \item $\bm c(\bm \xi) = \bm C\bm\xi$, $\bm d(\bm\xi) = \bm D\bm\xi$, and $\bm h(\bm\xi) = \bm H\bm\xi$ for some matrices $\bm C$, $\bm D$, %
              and $\bm H$.
        \item The matrices, $\bm C$, $\bm D$, $\bm T$, $\bm W$, $\bm H$, are integer-valued.
        \item There exists $\bm{x} \in \mathcal X$ for which $\sup \big\{\mathcal{Q}(\bm{x}, \bm{\xi}) : \bm{\xi} \in \Xi \big\} < +\infty$.
    \end{enumerate}
    Then, there exists a finite $\bar{\lambda} \geq 0$ that is computable in polynomial time in the input data, $\bm C$, $\bm D$, $\bm T$, $\bm W$, $\bm H$, and the bounds ${\bm x}^\ell$, ${\bm x}^u$, ${\bm y}^\ell$, ${\bm y}^u$, such that
    \begin{equation*}
        \inf_{\bm x\in \mathcal X} \sup_{\bm\xi\in\Xi} \mathcal Q(\bm x, \bm \xi)
        =
        \inf_{\bm x\in \mathcal X} \sup_{\bm\xi\in\Xi} \mathcal L(\bm x, \bm \xi, \bar{\lambda}).
    \end{equation*}
\end{theorem}

We provide some intuition behind the proof of this theorem, before presenting it. For simplicity, consider first the case in which the second-stage decisions $\bm y$ are continuous. In this setting, it is relatively easy to show that $\lambda$ is equivalent to the dual variable of the constraint, ``$\phi(\bm\xi, \bm z) \le 0$'' which, in turn, enforces \mbox{``$\bm \xi = \bm z$''}. However, this multiplier is a function of the first-stage decisions $\bm x$ and uncertain parameter realization $\bm \xi$; denote it as $\lambda^*(\bm x,\bm\xi)$. The key idea is that any value $\lambda \ge \lambda^*(\bm x, \bm\xi)$ is also an optimal Lagrange multiplier. Theorefore, our main argument relies on upper bounding the optimal dual variables of a parameterized linear program. For this, we use classical results from parametric optimization, specifically Lemma~\ref{lemma:buchheim}, which allows us to obtain an upper bound on the optimal multiplier, independent of~$\bm x$ and~$\bm\xi$. Moreover, this bound is easily shown to be computable in polynomial time. When the second-stage problem is an MILP, we rely on similar arguments except that we consider an LP that is parameterized not only by $\bm x$ and $\bm\xi$, but also by the discrete second-stage decisions $\bm y_\mathrm{d}$. We note that the integrality requirements on the matrices $\bm C, \bm D, \bm T, \bm W$ and $\bm H$ are necessary to obtain a formal statement of complexity. %
	We now present a formal proof of the theorem.

\begin{proof}[Proof of Theorem~\ref{theorem:polynomial_time_lambda}]
    Let $U$ denote any finite upper bound on the optimal value of problem~\ref{eq:two_stage_ro_general},
    \begin{equation*}
        \inf_{\bm{x} \in \mathcal{X}} \sup_{\bm{\xi} \in \Xi} \, \mathcal{Q}(\bm{x}, \bm{\xi}) \le U.
    \end{equation*}
    Note that $U$ exists due to the last condition in the hypothesis of the theorem. Moreover, it can be easily shown, similar to~\citet[Theorem~5]{subramanyam2022lagrangian}, that under the stated hypotheses, $U$ can be computed in polynomial time.

    Now, denote the feasible region of $\mathcal L(\bm x, \bm\xi,\lambda ; {\bm y}_\mathrm{d})$, defined in~\eqref{eq:L_fixed_yd}, as
    \begin{equation*}
        \Pi(\bm x, {\bm y}_\mathrm{d}) = \left\{ (\bm y_\mathrm{c}, \bm z)\in\mathbb R^{nc_2}_{+} \times[0,1]^{n_p} : \bm{T}\bm{x} + \bm{W}_\mathrm{c} \bm{y}_\mathrm{c} + \bm W_\mathrm{d} {\bm y}_\mathrm{d} \geq \bm h (\bm z) \right\}.
    \end{equation*}
    The majority of the proof is concerned with showing that one can compute a finite $\bar{\lambda}$ in polynomial time, satisfying the following relationship for all $\bm{x} \in \mathcal X$, $\bm \xi \in \Xi$ and ${\bm y}_\mathrm{d} \in \mathcal Y_\mathrm{d}$:
    \begin{equation}\label{eq:complexity:to_show}
        \mathcal{L}(\bm{x}, \bm{\xi}, \bar{\lambda} ; {\bm y}_\mathrm{d})
        \begin{cases}
            = \mathcal Q(\bm{x}, \bm{\xi}; {\bm y}_\mathrm{d}) &
            \text{if } \mathcal Q(\bm{x}, \bm{\xi}; {\bm y}_\mathrm{d}) < +\infty \text{ or } \Pi(\bm x, {\bm y}_\mathrm{d}) = \emptyset, \\
            \geq U                                             &
            \text{otherwise}.
        \end{cases}
    \end{equation}
    Supposing for the moment that this can be done, observe that we immediately obtain the desired equation stated in the theorem, since~\eqref{eq:complexity:to_show} implies
    \begin{equation*}
        \begin{aligned}
            \inf_{\bm x\in \mathcal X} \sup_{\bm\xi\in\Xi} \mathcal{L}(\bm{x}, \bm{\xi}, \bar{\lambda})
             & =
            \inf_{\bm x\in \mathcal X} \sup_{\bm\xi\in\Xi} \inf_{\bm y_\mathrm{d} \in \mathcal Y_\mathrm{d}} \mathcal{L}(\bm{x}, \bm{\xi}, \bar{\lambda} ; \bm y_\mathrm{d})
            \\
             & =
            \inf_{\bm x\in \mathcal X} \sup_{\bm\xi\in\Xi} \inf_{\bm y_\mathrm{d} \in \mathcal Y_\mathrm{d}} Q(\bm{x}, \bm{\xi}; \bm y_\mathrm{d})
            \\
             & =
            \inf_{\bm x\in \mathcal X} \sup_{\bm\xi\in\Xi} \mathcal{Q}(\bm{x}, \bm{\xi}),
        \end{aligned}
    \end{equation*}
    where we used the fact that $U$ is an upper bound on the optimal value of~\ref{eq:two_stage_ro_general}.

    We now proceed to establish the validity of~\eqref{eq:complexity:to_show}.
    Fix $\bm x\in \mathcal X$, $\bm \xi\in\Xi$, and ${\bm y}_\mathrm{d}\in \mathcal Y_\mathrm{d}$.
    The key idea is that $\mathcal Q(\bm x, \bm\xi; {\bm y}_\mathrm{d})$ can be equivalently written as:
    \begin{equation*}
        \mathcal{Q}(\bm{x}, \bm{\xi}; {\bm y}_\mathrm{d}) =
        \left[
            \begin{aligned}
                \mathop{\text{minimize}}_{\bm{y}_\mathrm{c} \in \mathbb{R}^{nc_2}_{+},  \bm{z} \in \mathbb{R}^{n_p}_{+}}\;\;
                                             & \bm{c}(\bm{\xi})^\top \bm{x}  + \bm{d}_\mathrm{c}(\bm{\xi})^\top \bm{y}_\mathrm{c} + \bm{d}_\mathrm{d}(\bm{\xi})^\top {\bm y}_\mathrm{d} \\
                \text{subject to} \;\;\;\;\; & \bm{T}\bm{x} + \bm{W}_\mathrm{c} \bm{y}_\mathrm{c} + \bm W_\mathrm{d} {\bm y}_\mathrm{d} \geq \bm{H} \bm z,                              \\
                                             & \bm z \le \one, \;\; \phi(\bm z, \bm \xi) \le 0.
            \end{aligned}
            \right]
    \end{equation*}
    Similarly, since $\phi(\bm z, \bm \xi) \geq 0$ from \citet[Lemma~1]{subramanyam2022lagrangian}, we note that $\mathcal{L}(\bm{x}, \bm{\xi}, \lambda; {\bm y}_\mathrm{d})$ can be equivalently written as:
    \begin{equation*}
        \mathcal{L}(\bm{x}, \bm{\xi}, \lambda; {\bm y}_\mathrm{d})
        =\left[
            \begin{aligned}
                \mathop{\text{minimize}}_{\substack{\bm{y}_\mathrm{c} \in \mathbb{R}^{nc_2}_{+},  \bm{z} \in \mathbb{R}^{n_p}_{+}                                 \\w \in \mathbb R_{+}}}\;\;
                                             &
                \bm{c}(\bm{\xi})^\top \bm{x}  + \bm{d}_\mathrm{c}(\bm{\xi})^\top \bm{y}_\mathrm{c} + \bm{d}_\mathrm{d}(\bm\xi)^\top{\bm y}_\mathrm{d} + \lambda w \\
                \text{subject to} \;\;\;\;\; & \bm{T}\bm{x} + \bm{W}_\mathrm{c} \bm{y}_\mathrm{c} + \bm W_\mathrm{d}{\bm y}_\mathrm{d} \geq \bm{H} \bm{z},        \\
                                             & \bm{z} \leq \one, \;\; \phi(\bm{z}, \bm{\xi}) \le w.
            \end{aligned}
            \right]
    \end{equation*}
    We now distinguish two cases.
    \begin{enumerate}
        \item
              Suppose $\mathcal{Q}(\bm{x}, \bm{\xi}; {\bm y}_\mathrm{d}) < +\infty$.

              Together with \citet[Assumption~A1]{subramanyam2022lagrangian}, this means that $\mathcal{Q}(\bm{x}, \bm{\xi}; {\bm y}_\mathrm{d})$ is finite. Strong linear programming duality then implies that it can be equivalently written as the (finite) optimal value of the following problem, where we have also expanded $\phi(\bm z, \bm \xi)$ using its definition~\eqref{eq:penalty_general}.
              \begin{equation*}
                  \begin{aligned}
                      \mathop{\text{maximize}}_{\bm{\mu} \in \mathbb R_{+}^m, \bm{\beta} \in \mathbb R_{+}^{n_p}, \alpha \in \mathbb{R}_{+}} \;\;
                                                         &
                      \bm{c}(\bm{\xi})^\top \bm{x}  + \bm{d}_\mathrm{d}(\bm\xi)^\top{\bm y}_\mathrm{d}
                      + ( - \bm T\bm x - \bm W_\mathrm{d}{\bm y}_\mathrm{d} )^\top\bm\mu  + \one^\top(\alpha \bm\xi - \bm\beta)
                      \\
                      \text{subject to} \;\;\;\;\;\;\;\; & \bm W_\mathrm{c}^\top \bm\mu \leq \bm d_\mathrm{c}(\bm\xi),     \\
                                                         & (2\bm\xi - \one)\alpha - \bm H^\top\bm\mu - \bm\beta \le \bm 0.
                  \end{aligned}
              \end{equation*}
              The dual of $\mathcal{L}(\bm{x}, \bm{\xi}, \lambda; {\bm y}_\mathrm{d})$ is identical to the above problem with the additional constraint, $\alpha \leq \lambda$.
              Therefore, one can ensure~\eqref{eq:complexity:to_show} by choosing $\bar{\lambda}$ to be any upper bound on an optimal value of $\alpha$.
              Without loss of generality (e.g., by converting the dual problem to standard form), there exists an optimal solution to the dual problem that lies at a vertex of the polyhedron defining its feasible region.
              Lemma~\ref{lemma:buchheim} ensures that all entries, and in particular $\alpha$, of such a vertex can be upper bounded by
              \begin{equation*}
                  (nc_2 + n_p)! \lVert \bm d_\mathrm{c}(\bm\xi) \rVert \max\{ \lVert \bm W_\mathrm{c} \rVert, \lVert \bm H \rVert, \lVert 2\bm\xi - \one \rVert, 1 \}^{ nc_2 + n_p - 1 }.
              \end{equation*}
              Observe now that $\lVert \bm d_\mathrm{c}(\bm\xi) \rVert = \lVert \bm D_\mathrm{c} \bm \xi \rVert = \lVert \bm D_\mathrm{c} \bm \xi \rVert_{\infty} \leq \lVert \bm D_\mathrm{c} \rVert_{\infty} \lVert \bm \xi \rVert_{\infty} \leq \lVert \bm D_\mathrm{c} \rVert_{\infty}$, since $\bm \xi \in \{0, 1\}^{n_p}$ and where $\lVert \bm D_\mathrm{c} \rVert_{\infty}$ denotes the vector-induced matrix norm of $\bm D_\mathrm{c}$.
              Also, $\lVert 2\bm\xi - \one \rVert \leq 1$.
              One can therefore choose $\bar{\lambda}$ as follows:
              \begin{equation*}
                  \hat{\lambda} \geq (nc_2 + n_p)! \lVert \bm D_\mathrm{c} \rVert_\infty \max\{ \lVert \bm W_\mathrm{c} \rVert, \lVert \bm H \rVert, 1 \}^{ nc_2 + n_p - 1 }.
              \end{equation*}
              Note that this bound does not depend on the chosen $\bm x$, $\bm \xi$ or ${\bm y}_\mathrm{d}$ and is computable in polynomial time in the input data.

        \item
              Suppose now $\mathcal{Q}(\bm{x}, \bm{\xi}; {\bm y}_\mathrm{d}) = +\infty$.
              We again consider two cases.
              \begin{enumerate}
                  \item
                        Suppose $\Pi(\bm x, {\bm y}_\mathrm{d}) = \emptyset$. Then, $\mathcal L(\bm x, \bm\xi,\lambda ; {\bm y}_\mathrm{d}) = +\infty$ for any $\lambda \ge 0$. Hence, one can safely choose any $\bar{\lambda} \geq 0$ to achieve~\eqref{eq:complexity:to_show}.
                  \item
                        Suppose $\Pi(\bm x, {\bm y}_\mathrm{d}) \neq \emptyset$.
                        Using strong Lagrangian duality~\eqref{eq:strong_duality_general}, %
                        it follows that for any finite $V$, there must exist a sufficiently large yet finite $\lambda$ satisfying $\mathcal L(\bm x, \bm\xi, \lambda ; {\bm y}_\mathrm{d}) \geq V$.
                        Since $\mathcal L(\bm x, \bm\xi, \lambda ; {\bm y}_\mathrm{d})$ is also finite (because $\Pi(\bm x, {\bm y}_\mathrm{d}) \neq \emptyset$),
                        finding $\lambda$ satisfying $\mathcal L(\bm x, \bm\xi, \lambda ; {\bm y}_\mathrm{d}) \geq V$ is equivalent to replacing $\mathcal L(\bm x, \bm\xi,\lambda ; {\bm y}_\mathrm{d})$ by its (finite) dual and finding a feasible vector $(\bm{\mu}, \bm{\beta}, \alpha, \lambda) \in \mathbb R_{+}^m \times \mathbb R_{+}^{n_p} \times \mathbb{R}_{+} \times \mathbb{R}_{+}$ satisfying the linear constraints:
                        \begin{align*}
                             &
                            (- \bm T\bm x - \bm W_\mathrm{d}{\bm y}_\mathrm{d})^\top\bm\mu  + \one^\top(\alpha \bm  \xi - \bm \beta)
                            \geq V - \bm{c}(\bm{\xi})^\top \bm{x}  - \bm{d}_\mathrm{d}(\bm\xi)^\top{\bm y}_\mathrm{d}, \\
                             & \bm W_\mathrm{c}^\top \bm\mu \leq \bm d_\mathrm{c}(\bm\xi),                             \\
                             & (2\bm\xi - \bm e)\alpha -\bm H^\top\bm\mu -\bm\beta \le \bm 0,                          \\
                             & \alpha - \lambda \leq 0.
                        \end{align*}
                        We can then achieve~\eqref{eq:complexity:to_show} by choosing $V = U$ and choosing $\bar{\lambda}$ to be an upper bound on the entry $\lambda$ of a feasible solution of the above inequality system.
                        To that end, it suffices to bound the vertices of the (equivalent standard form) polyhedron defined by the above inequalities in variables $(\bm\mu,\bm\beta,\alpha,\lambda)$.
                        Lemma~\ref{lemma:buchheim} can be used to compute such a bound in polynomial time.
                        In particular, it can be shown that one can choose
                        \begin{equation*}
                            \bar{\lambda} \geq ( nc_\mathrm{2} + n_p + 2)! \cdot \theta_1^{nc_\mathrm{2} + n_p + 2} \cdot \theta_2 \cdot \theta_3^{nc_2 + n_p + 1},
                        \end{equation*}
                        with $\theta_1 = \max\{ \lVert \bm x^\ell \rVert, \lVert \bm x^u \rVert, \lVert \bm y^\ell_\mathrm{d} \rVert, \lVert \bm y^u_\mathrm{d} \rVert, 1 \}$, $\theta_2 = \max\{ U + \lVert \bm{C}\rVert_\infty + \lVert \bm{D}_\mathrm{d}\rVert_\infty, \lVert \bm D_\mathrm{c} \rVert_\infty \}$ and $\theta_3 = \max\{ \lVert \bm W_\mathrm{d} \rVert_\infty + \lVert \bm T \rVert_\infty, \lVert \bm W_\mathrm{c} \rVert, \lVert \bm H \rVert, 1 \}$.
                        As before, note that this bound does not depend on the chosen $\bm x$, $\bm \xi$ or ${\bm y}_\mathrm{d}$.
              \end{enumerate}
    \end{enumerate}
    The validity of~\eqref{eq:complexity:to_show} now simply follows by defining $\bar{\lambda}$ to be the maximum of the three bounds obtained in the three disjunctions.%
\end{proof}

The proof of Theorem~\ref{theorem:polynomial_time_lambda} is constructive: it shows how one can derive a closed-form upper bound on the optimal Lagrange multiplier. However, the bound is too large to be useful in practice; therefore, the theorem should only be viewed as providing a complexity-theoretic result rather than a practical solution.  Even in problems where all first- and second-stage decisions are continuous, deriving a ``small'' upper bound on the optimal multiplier \emph{ex ante} can be practically challenging; we refer to the literature on bilevel optimization \citep{Kleinert2020,buchheim2023bilevel} for related work. %

\subsection{Algorithmic Modifications}
\label{sec:algo_modifications}

Although one can compute an optimal multiplier in polynomial time, it does not preclude the possibility that verifying optimality of a given multiplier cannot be done in polynomial time. To that end, we propose practical yet simple strategies in cases where a candidate multiplier may be suboptimal. In particular, we propose algorithmic modifications to the Benders decomposition and column-and-constraint generation algorithms originally presented in~\citet{subramanyam2022lagrangian}. Notably, these modifications enable us to recover exact methods for general instances of problems~\ref{eq:two_stage_ro_general} and~\ref{eq:two_stage_ro_indicator}, even if they do not satisfy any of the conditions presented in Theorems~\ref{theorem:optimal_multiplier_corrected}--\ref{theorem:polynomial_time_lambda}.

Classical versions of the Benders and column-and-constraint generation algorithms obtain upper bounds on the two-stage problem by solving $\sup \{ \mathcal{Q}(\bm{x}, \bm{\xi}) : \bm{\xi} \in \Xi \}$ for some fixed $\bm{x} \in \mathcal{X}$. Instead, \citet{subramanyam2022lagrangian} propose to solve $\sup \{ \mathcal{L}(\bm{x}, \bm{\xi}, \lambda) : \bm{\xi} \in \Xi \}$, where the second-stage value function is replaced by the Lagrangian function.
In the absence of sufficient conditions that ensure optimality of $\lambda$, it may be possible %
that the calculated upper bounds are no longer rigorous.

This issue can be addressed by a simple modification.
The key idea is to use \citet[Theorem~3]{subramanyam2022lagrangian} which provides \emph{necessary conditions} for the optimality of a Lagrange multiplier.
This theorem is exploited in Algorithms~4 and~5 of the original paper, proposed for problems~\ref{eq:two_stage_ro_general} and~\ref{eq:two_stage_ro_indicator}, respectively, which output either an uncertain parameter realization that makes the second-stage problem infeasible or a Lagrange multiplier which satisfies the necessary conditions.
These algorithms are then embedded within the corresponding Benders and column-and-constraint generation algorithms.
To ensure that the latter do not terminate incorrectly, the proposed modification indirectly verifies optimality of the calculated Lagrange multiplier \emph{ex post}.

To simplify exposition and maintain consistency with the original paper, we first illustrate this modification in the context of the Benders decomposition and column-and-constraint generation algorithms for solving formulation~\ref{eq:two_stage_ro_indicator} with continuous second-stage decisions ($\mathcal Y = \mathbb{R}_{+}^{n_2}$). In particular, this problem structure arises in the first two numerical experiments of the original paper. We then present the modifications for the more general formulation~\ref{eq:two_stage_ro_general} with mixed-integer second-stage decisions, which arises in the third experiment of the original paper as well as the counterexample in Section~\ref{sec:counterexample}.
We omit presenting modifications for problem~\ref{eq:two_stage_ro_indicator} with mixed-integer second-stage decisions since it is very similar to the latter.
We finally close the paper with a discussion about the computational efficiency of the proposed modifications.

\subsubsection{Modifications for Problem~\ref{eq:two_stage_ro_indicator}}
The updated versions of the Benders and column-and-constraint generation algorithms are shown in Algorithm~\ref{algo:indicator:updated}. %
The algorithm indirectly checks if the estimated upper bound is less than the optimal value of the original problem, by solving %
\begin{equation}\label{eq:worst_case_problem_indicator} %
    \begin{aligned}
        \mathop{\text{maximize}}_{\substack{\bm{\xi} \in \Xi, \bm{\mu} \in \mathbb{R}^m_{+} \\ \bm{\rho} \in \mathbb{R}^{n_p}_{+}, \bm{\nu} \in \mathbb{R}^{n_p}_{+}}} \;\;
         &
        \bm{c}(\bm{\xi})^\top\bm{x} +
        (\bm{h}_0 - \bm{T}\bm{x})^\top \bm{\psi}(\bm \mu, \bm\rho, \bm \nu)                 \\
        \text{subject to} \;\;\;
         & \bm{W}^\top \bm{\psi}(\bm \mu, \bm\rho, \bm \nu) \leq \bm{d}(\bm{\xi}),          \\
         & \xi_j = 0 \implies \rho_j = 0, \; j \in [n_p],                                   \\
         & \xi_j = 1 \implies \nu_j = 0, \; j \in [n_p],
    \end{aligned}
\end{equation}
where $\bm \psi$ is defined as follows:
\[
    \bm{\psi}(\bm \mu, \bm\rho, \bm \nu) = \bm \mu - \sum_{j \in [n_p]} \sum_{i \in \mathcal{I}_j^1} \rho_j \one_i - \sum_{j \in [n_p]}\sum_{i \in \mathcal{I}_j^0} \nu_j \one_i.
\]

\begin{algorithm}[!htbp]
    \caption{Updated Benders decomposition and column-and-constraint generation algorithms to solve~\ref{eq:two_stage_ro_indicator} when $\mathcal Y = \mathbb{R}_{+}^{n_2}$}
    \label{algo:indicator:updated}
    \begin{algorithmic}
        \STATE Benders: Run all lines of \citet[Algorithm~6]{subramanyam2022lagrangian}
        \STATE Column-and-constraint generation: Run all lines of \citet[Algorithm~7]{subramanyam2022lagrangian}
        \IF {$\hat{\bm{x}} \neq \emptyset$}
        \STATE Set $Z$ and $(\hat{\bm\rho}, \hat{\bm \nu})$ as the optimal value and (projected) solution of~\eqref{eq:worst_case_problem_indicator} (at $\bm{x} = \hat{\bm{x}}$)
        \IF {$UB < Z$}
        \STATE Update $UB \gets Z$ and $\lambda \gets \max\left\{ \lVert \hat{\bm \rho} \rVert_\infty, \lVert \hat{\bm\nu} \rVert_\infty \right\}$
        \STATE Go to line~2 of the original algorithm %
        \ENDIF
        \ENDIF
    \end{algorithmic}
\end{algorithm}

The following theorem rigorously justifies the proposed modification. It shows that an optimal value of the Lagrange multiplier can be easily computed given an optimal solution of problem~\eqref{eq:worst_case_problem_indicator}.
\begin{theorem}\label{theorem:optimal_multiplier_ex_post_indicator}
    Suppose $\mathcal Y = \mathbb R^{n_2}_{+}$ and $\bm x \in \mathcal X$ is any feasible first-stage decision in problem~\ref{eq:two_stage_ro_indicator}.
    Let $(\hat{\bm\xi}, \hat{\bm\mu}, \hat{\bm\rho}, \hat{\bm\nu})$ denote an optimal solution of problem~\eqref{eq:worst_case_problem_indicator}.
    Then, $\bar \lambda = \max\left\{ \lVert \hat{\bm \rho} \rVert_\infty, \lVert \hat{\bm\nu} \rVert_\infty \right\}$ is an optimal multiplier satisfying
    \begin{equation*}
        \sup_{\bm\xi\in\Xi} \mathcal Q_\mathcal{I}(\bm x, \bm \xi)
        =
        \sup_{\bm\xi\in\Xi} \mathcal L_\mathcal{I}(\bm x, \bm \xi, \bar{\lambda}).
    \end{equation*}
\end{theorem}
\begin{proof}
    Let $Z$ denote the optimal value of problem~\eqref{eq:worst_case_problem_indicator}.
    Now, consider problem~\eqref{eq:worst_case_problem_indicator}, where $\bm \xi$, $\bm \rho$ and $\bm \nu$ are fixed to $\hat{\bm \xi}$, $\hat{\bm \rho}$ and $\hat{\bm \nu}$, respectively.
    The resulting problem can be equivalently written as a %
    linear maximization problem over $\bm \mu \in \mathbb{R}^m_{+}$. %
    Taking the linear programming dual of this problem yields %
    \begin{equation*}
        \begin{aligned}
            Z = %
            \inf_{\bm y \in \mathcal Y(\bm x)} \;\;
            \bm{c}(\hat{\bm{\xi}})^\top\bm{x} +\bm{d}(\hat{\bm{\xi}})^\top\bm{y} + \sum_{j \in [n_p]} \sum_{i \in \mathcal{I}_j^1} \hat{\rho}_j g_i(\bm x, \bm y) + \sum_{j \in [n_p]}\sum_{i \in \mathcal{I}_j^0} \hat{\nu}_j g_i(\bm x, \bm y) %
        \end{aligned}
    \end{equation*}
    where we used the definition of $\bm g(\bm x, \bm y) = \bm T \bm x + \bm W \bm y - \bm h_0$ and define $\mathcal{Y}(\bm x) = \{\bm y \in \mathcal Y : \bm g (\bm x, \bm y) \geq \bm 0\}$.
    The definition of $\bar{\lambda}$ along with the indicator constraints in problem~\eqref{eq:worst_case_problem_indicator} imply that the inequalities, $\hat{\rho}_j \leq \bar{\lambda} \hat{\xi}_j$ and $\hat{\nu}_j \leq \bar{\lambda} (1 - \hat{\xi}_j)$, hold for all $j \in [n_p]$.
    Substituting these inequalities above gives:
    \begin{align*}
        Z
         & \leq
        \inf_{\bm y \in \mathcal Y(\bm x)} \;\;
        \bm{c}(\hat{\bm{\xi}})^\top\bm{x} +\bm{d}(\hat{\bm{\xi}})^\top\bm{y} +  \sum_{j \in [n_p]} \sum_{i \in \mathcal{I}_j^1} \bar{\lambda} \hat{\xi}_j g_i(\bm x, \bm y) + \sum_{j \in [n_p]}\sum_{i \in \mathcal{I}_j^0} \bar{\lambda}(1 - \hat{\xi}_j) g_i(\bm x, \bm y)
        \\
         & =
        \inf_{\bm y \in \mathcal Y(\bm x)} \;\;
        \bm{c}(\hat{\bm{\xi}})^\top\bm{x} +\bm{d}(\hat{\bm{\xi}})^\top\bm{y} +  \bar{\lambda} \phi_\mathcal{I}(\bm x, \bm y, \hat{\bm\xi})
        \\
         & \leq
        \sup_{\bm \xi \in \Xi} \; \inf_{\bm y \in \mathcal Y(\bm x)} \;\;
        \bm{c}(\bm{\xi})^\top\bm{x} +\bm{d}(\bm{\xi})^\top\bm{y} +  \bar{\lambda} \phi_\mathcal{I}(\bm x, \bm y, \bm\xi)
        \\
         & = \sup_{\bm\xi\in\Xi} \mathcal L_\mathcal{I}(\bm x, \bm \xi, \bar{\lambda})
        \\
         & \leq \sup_{\bm\xi\in\Xi} \mathcal Q_\mathcal{I}(\bm x, \bm \xi),
    \end{align*}
    where the first equality holds by definition of $\phi_\mathcal{I}$, the next inequality follows by taking the supremum of the previous expression over $\bm \xi \in \Xi$, the second equality follows by definition of $\mathcal L_\mathcal{I}(\bm x, \bm \xi, \bar{\lambda})$, and the last inequality follows by weak duality~\citep[Lemma~2]{subramanyam2022lagrangian}.

    Now, strong duality~\citep[Theorem~1]{subramanyam2022lagrangian} implies that
    \begin{align*}
        \sup_{\bm\xi\in\Xi} \mathcal Q_\mathcal{I}(\bm x, \bm \xi)
         & =
        \sup_{\lambda \geq 0} \sup_{\bm\xi\in\Xi} \mathcal L_\mathcal{I}(\bm x, \bm \xi, \lambda)
        \\
         & =
        \sup_{\bm \xi \in \Xi} \; \inf_{\bm y \in \mathcal Y(\bm x)} \;\;
        \bm{c}(\bm{\xi})^\top\bm{x} +\bm{d}(\bm{\xi})^\top\bm{y} +  \lambda \phi_\mathcal{I}(\bm x, \bm y, \bm\xi)
        \\
         & =
        \left[\begin{aligned}
                      \mathop{\text{maximize}}_{\lambda \geq 0, \bm{\xi} \in \Xi, \bm{\mu} \in \mathbb{R}^m_{+}} \;\;
                       &
                      \bm{c}(\bm{\xi})^\top\bm{x} +
                      (\bm{h}_0 - \bm{T}\bm{x})^\top \bm{\psi}\left( \bm \mu, \lambda\bm\xi, \lambda(\one - \bm \xi) \right)      \\
                      \text{subject to} \;\;\;\;\;
                       & \bm{W}^\top \bm{\psi}\left( \bm \mu, \lambda\bm\xi, \lambda(\one - \bm \xi)\right) \leq \bm{d}(\bm{\xi})
                  \end{aligned}\right],
    \end{align*}
    where the last equality follows by linear programming duality.
    Let $(\tilde{\lambda}, \tilde{\bm \xi}, \tilde{\bm \mu})$ be an optimal solution of the above (bilinear) optimization problem.
    Define $\tilde{\bm \rho} = \tilde{\lambda} \tilde{\bm \xi}$ and $\tilde{\bm \nu} = \tilde{\lambda} (\one - \tilde{\bm \xi})$.
    Then, it can be readily verified that $(\tilde{\bm\xi}, \tilde{\bm\mu}, \tilde{\bm\rho}, \tilde{\bm\nu})$ is a feasible solution in problem~\eqref{eq:worst_case_problem_indicator} that achieves an objective value equal to the optimal value of the above bilinear problem. This implies
    \begin{align*}
        \sup_{\bm\xi\in\Xi} \mathcal Q_\mathcal{I}(\bm x, \bm \xi) \leq Z,
    \end{align*}
    which taken together with our previously established inequality,
    \begin{align*}
        Z
        \leq \sup_{\bm\xi\in\Xi} \mathcal L_\mathcal{I}(\bm x, \bm \xi, \bar{\lambda})
        \leq \sup_{\bm\xi\in\Xi} \mathcal Q_\mathcal{I}(\bm x, \bm \xi),
    \end{align*}
    proves the claimed result.%
\end{proof}

Algorithm~\ref{algo:indicator:updated} checks if the optimal value of problem~\eqref{eq:worst_case_problem_indicator} is larger than the final estimate $UB$, and then uses the result of Theorem~\ref{theorem:optimal_multiplier_ex_post_indicator} to initialize another run of the original procedure with the updated $\lambda$ and corrected $UB$.
In doing so, it retains all data structures without re-initializing them to be empty sets.
In particular, for the Benders algorithm, the feasibility and optimality sets, $\mathcal F$ and $\mathcal O$, are retained, and all previously generated Benders cuts are simply lifted with the updated value of $\lambda$.
Similarly, for the column-and-constraint generation algorithm, the set of enumerated uncertain parameters $\mathcal R$ is retained.
It is crucial to higlight that in both algorithms, all previously generated constraints continue to remain valid, and therefore, they always provide rigorous lower bounds on the optimal value of the original two-stage problem.
In particular, this is also true for Benders cuts generated using suboptimal $\lambda$ values.
Formally, this is because of weak duality~\citep[Lemma~2]{subramanyam2022lagrangian}, which implies:
\[
    \inf_{\bm{x} \in \mathcal{X}} \sup_{\bm{\xi} \in \Xi} \mathcal{Q}_\mathcal{I}(\bm{x}, \bm{\xi})
    \geq
    \inf_{\bm{x} \in \mathcal{X}} \sup_{\bm{\xi} \in \Xi} \mathcal{L}_\mathcal{I}(\bm{x}, \bm{\xi}, \lambda)
    \; \text{ for all } \lambda \geq 0.
\]

We note that if Algorithm~5 of the original paper (which is invoked within the original Algorithms~6 and~7) outputs an optimal multiplier, then problem~\eqref{eq:worst_case_problem_indicator} is solved at most once.
This is important for reasons of computational efficiency, which we discuss at the end of the paper.

\subsubsection{Modifications for Problem~\ref{eq:two_stage_ro_general}}
We now consider the general version of~\ref{eq:two_stage_ro_general} with mixed-integer second-stage decisions.
Keeping in line with the original paper, we focus only on the column-and-constraint generation algorithm.
We first provide an updated version of the original method from~\citet{subramanyam2022lagrangian} in Algorithm~\ref{algo:ccg:inner}.
Whereas the original method assumes that sufficient conditions for optimality of $\lambda$ are already satisfied, Algorithm~\ref{algo:ccg:inner} does not make any such assumption.

\begin{algorithm}[!htbp]
    \caption{Updated version of \citet[Algorithm~8]{subramanyam2022lagrangian}}
    \label{algo:ccg:inner}
    \begin{algorithmic}[1]
        \renewcommand{\algorithmicrequire}{\textbf{Input:}}
        \renewcommand{\algorithmicensure}{\textbf{Output:}}
        \REQUIRE $\bm{x} \in \mathcal{X}$, $\lambda^0 > 0$
        \ENSURE %
        Either $\hat{\bm{\xi}} \in \Xi: \mathcal{Q}(\bm{x}, \hat{\bm{\xi}}) = +\infty$, $\hat{\lambda} = \lambda^0$ or $\hat{\bm{\xi}}, \hat{\lambda}$ satisfying conditions of \citet[Theorem~3]{subramanyam2022lagrangian}
        \STATE Initialize $\hat{\bm{\xi}} \in \Xi$ (arbitrary), $LB = -\infty$, $UB = +\infty$, $\mathcal{D} = \emptyset$, $\hat \lambda = \lambda^0$.
        \REPEAT
        \STATE Set $UB$ and $\hat{\bm{\xi}}$ as optimal value and (projected) solution of~\eqref{eq:worst_case_problem_discrete_lagrangian_duality} with $(\tau, \lambda) = (0, 1)$
        \label{algo:ccg:inner:feasibility:ub-update}
        \STATE Set
        $
            (\hat{\bm{y}}, \hat{\bm{z}}, \hat{\bm{\sigma}}) \in \mathop{\arg\min}\limits_{(\bm{y}, \bm{z}, \bm{\sigma}) \in \mathcal{Y} \times [0, 1]^{n_p} \times \mathbb{R}^m_{+}}
            \left\{\one^\top \bm{\sigma} + \phi(\bm{z}, \hat{\bm{\xi}}) : \bm{T}\bm{x} + \bm{W} \bm{y} + \bm{\sigma} \geq \bm{h}(\bm{z}) \right\}
        $
        \label{algo:ccg:inner:feasibility:lb-update}
        \STATE Update $\mathcal{D} \gets \mathcal{D} \cup \{\hat{\bm{y}}_\mathrm{d}\}$
        \STATE Update $LB = \one^\top \hat{\bm{\sigma}} + \phi(\hat{\bm{z}}, \hat{\bm{\xi}})$
        \UNTIL {$LB > 0$ or $UB = 0$}
        \IF {$UB = 0$}
        \STATE Set $LB = -\infty$
        \REPEAT
        \STATE Update $\hat{\lambda} \gets \hat{\lambda}/2$
        \label{algo:ccg:inner:optimality:lambda-initialize}
        \REPEAT
        \STATE Update $\hat{\lambda} \gets 2\hat{\lambda}$
        \STATE Set $UB$, $\tilde{\bm{\xi}}$ as optimal value and
        (projected) solution of~\eqref{eq:worst_case_problem_discrete_lagrangian_duality} with $(\tau, \lambda) = (1, \hat{\lambda})$
        \label{algo:ccg:inner:optimality:ub-update}
        \STATE Set
        $
            (\hat{\bm{y}}, \hat{\bm{z}}) \in \mathop{\arg\min}\limits_{(\bm{y}, \bm z) \in \mathcal{Y} \times [0, 1]^{n_p}}
            \left\{\bm c(\tilde{\bm \xi})^\top \bm x + \bm{d}(\tilde{\bm{\xi}})^\top \bm{y} + \hat \lambda \phi(\bm z, \tilde{\bm \xi}): \bm{T}\bm{x} + \bm{W} \bm{y} \geq \bm{h}(\bm z) \right\}
        $
        \STATE Update $\mathcal{D} \gets \mathcal{D} \cup \{\hat{\bm{y}}_\mathrm{d}\}$
        \UNTIL $\phi(\hat{\bm z}, \tilde{\bm \xi}) = 0$
        \label{algo:ccg:inner:optimality:necessary-condition}
        \IIf {$LB < \bm c(\tilde{\bm \xi})^\top \bm x + \bm{d}(\tilde{\bm{\xi}})^\top \hat{\bm{y}}$}
        update
        $LB \gets \bm c(\tilde{\bm \xi})^\top \bm x + \bm{d}(\tilde{\bm{\xi}})^\top \hat{\bm{y}}$
        and
        $\hat{\bm{\xi}} \gets \tilde{\bm{\xi}}$
        \EndIIf
        \label{algo:ccg:inner:optimality:lb-update}
        \UNTIL {$UB - LB \leq \epsilon$}
        \ENDIF
    \end{algorithmic}
\end{algorithm}

As described in the original paper, the key idea of the method is that for fixed $\bm{x} \in \mathcal X$, one can obtain a relaxation of the worst-case Lagrangian function, $\sup \{ \mathcal{L}(\bm{x}, \bm{\xi}, \lambda) : \bm{\xi} \in \Xi \}$, by enumerating the set of discrete second-stage decisions $\mathcal Y_\mathrm{d}$.
In particular, if $\mathcal D \subseteq \mathcal Y_\mathrm{d}$, then it follows that
\begin{equation*}
    \sup_{\bm \xi \in \Xi} \mathcal{L}(\bm{x}, \bm{\xi}, \lambda)
    =
    \sup_{\bm \xi \in \Xi} \inf_{\bm y_\mathrm{d} \in \mathcal Y_\mathrm{d}} \mathcal{L}(\bm{x}, \bm{\xi}, \lambda; \bm y_\mathrm{d})
    \leq
    \sup_{\bm \xi \in \Xi} \inf_{\bm y_\mathrm{d} \in \mathcal D} \mathcal{L}(\bm{x}, \bm{\xi}, \lambda; \bm y_\mathrm{d}).
\end{equation*}
Now, if $\bm y_\mathrm{d}^{(k)} \in \mathcal D$ denotes the $k^\text{th}$ element of $\mathcal{D}$, then it can be shown~\citep{zeng2013solving,subramanyam2022lagrangian} that the the following (with $\tau = 1$) is a reformulation of the problem on the right-hand side of the above inequality.
\begin{equation}\label{eq:worst_case_problem_discrete_lagrangian_duality} %
    \begin{aligned}
        \mathop{\text{maximize}}_{\eta, \bm \xi, \bm \mu, \bm \beta} \;\;
         &
        \eta
        \\
        \text{subject to} \;\;
         & \eta \in \mathbb R, \;\; \bm \xi \in \Xi, \\
         &
        \left.
        \begin{aligned}
             &
            \bm{\mu}^{(k)} \in \mathbb{R}^m_{+}, \;\; \bm{\beta}^{(k)} \in \mathbb{R}^{n_p}_{+},
            \\
             &
            \eta \leq
            \tau \bm{c}(\bm{\xi})^\top \bm{x}  + \tau \bm{d}_\mathrm{d}(\bm\xi)^\top \bm y_\mathrm{d}^{(k)} + \one^\top(\lambda \bm\xi - \bm\beta^{(k)})
            \\
             &
            \phantom{\eta \leq}
            + (\bm h_0 - \bm T\bm x - \bm W_\mathrm{d} \bm y_\mathrm{d}^{(k)} )^\top\bm\mu^{(k)},
            \\
             &
            \bm{W}_{\mathrm{c}}^\top \bm{\mu}^{(k)} \leq \tau \bm{d}_\mathrm{c}(\bm{\xi}), \; \; (1 - \tau) \bm \mu^{(k)} \leq \one,
            \\
             &
            2 \lambda\bm{\xi} - \bm{H}^\top \bm{\mu}^{(k)} - \bm{\beta}^{(k)} \leq \lambda\one,
        \end{aligned}\right\} \;\; k \in [|\mathcal D|].
    \end{aligned}
\end{equation}
The parameter $\tau \in \{0, 1\}$ serves purely to simplify notation.
Indeed, when $\tau = 0$, it can be shown that problem~\eqref{eq:worst_case_problem_discrete_lagrangian_duality} also provides an upper bound on the worst-case constraint violation function. %
The latter is equal to $0$ if and only if the first-stage decision $\bm x \in \mathcal X$ is feasible in problem~\ref{eq:two_stage_ro_general}, see~\citet{subramanyam2022lagrangian} for further details.

Algorithm~\ref{algo:ccg:inner} is almost identical to the original~\citet[Algorithm~8]{subramanyam2022lagrangian}, except for the inner loop based on the condition appearing in line~\ref{algo:ccg:inner:optimality:necessary-condition}. %
The latter enforces necessary conditions for optimality of $\lambda$ based on \citet[Theorem~3]{subramanyam2022lagrangian}.
Although these conditions are not sufficient, it is still possible that the final $\lambda$ is optimal, as illustrated in the following example.

\begin{example}[Counterexample revisited]
    We illustrate Algorithm~\ref{algo:ccg:inner} on the counterexample from Section~\ref{sec:counterexample}, with inputs $x = 0$ and $\lambda^0 = u(x) - \ell(x) = 1$.
    Let $y^{(k)} \in \{0, 1\}$ for $k \in \{1, 2\}$.
    Then, problem~\eqref{eq:worst_case_problem_discrete_lagrangian_duality} reduces to:
    \begin{equation*}
        \begin{aligned}
            \mathop{\text{maximize}}_{\xi, \eta, \mu, \beta} \;\;
             &
            \eta
            \\
            \text{subject to} \;\;
             & \xi \in \{0, 1\}, \;\; \eta \in \mathbb R, \\
             &
            \left.
            \begin{aligned}
                 &
                \mu^{(k)} \in \mathbb{R}_{+}, \;\; \beta^{(k)} \in \mathbb{R}_{+},
                \\
                 &
                \eta \leq
                -\tau y^{(k)} + \left( y^{(k)} - \frac32 \right) \mu^{(k)} - \beta^{(k)} + \lambda \xi,
                \\
                 &
                2 \lambda \xi - \mu^{(k)} - \beta^{(k)} \leq \lambda,
            \end{aligned}\right\} \;\; k \in \{1, 2\}.
        \end{aligned}
    \end{equation*}
    The various optimization problems solved in the algorithm often exhibit multiple optimal solutions.
    Straightforward computations (omitted for the sake of brevity) reveal that the algorithm will exit the first `repeat' loop with one of the following outcomes: $\mathcal D = \{0, 1\}$ after three iterations; $\mathcal D = \{0, 1\}$ after two iterations; $\mathcal D = \{0\}$ after two iterations. In all outcomes, $LB = UB = 0$, indicating that the original two-stage problem is feasible.

    \begin{enumerate}
        \item
              Suppose $D = \{0, 1\}$.
              Line~\ref{algo:ccg:inner:optimality:ub-update} yields $UB = -0.5$ and $\tilde \xi = 1$ as the unique solution.
              The next line yields $(\hat y, \hat z) = (1, 0.5)$ that is also unique.
              Since $\phi(\hat z, \hat \xi) = 0.5$, the algorithm updates $\hat \lambda = 2$.
              Line~\ref{algo:ccg:inner:optimality:ub-update} then yields $UB = 0$ and $\tilde \xi = 1$ (unique).
              The next line can yield one of two optimal solutions, $(\hat y, \hat z) = (0, 1)$ or $(\hat y, \hat z) = (1, 0.5)$.
              \begin{itemize}
                  \item If $(\hat y, \hat z) = (0, 1)$, then Line~\ref{algo:ccg:inner:optimality:lb-update} updates $LB = 0$. The algorithm stops with $\hat \lambda = 2$ and $LB = UB = 0$.
                  \item If $(\hat y, \hat z) = (1, 0.5)$, then since $\phi(\hat z, \hat \xi) = 0.5$, the algorithm updates $\hat \lambda = 4$.
                        Line~\ref{algo:ccg:inner:optimality:ub-update} yields $UB = 0$ and $\tilde \xi = 1$ (unique).
                        The next line yields $(\hat y, \hat z) = (0, 1)$ as the unique solution.
                        Line~\ref{algo:ccg:inner:optimality:lb-update} updates $LB = 0$.
                        The algorithm stops with $\hat \lambda = 4$ and $LB = UB = 0$.
              \end{itemize}

        \item
              Suppose $D = \{0\}$.
              Line~\ref{algo:ccg:inner:optimality:ub-update} yields $UB = 0$. Also, the optimal $\tilde \xi \in \{0, 1\}$.
              \begin{itemize}
                  \item If $\tilde \xi = 0$, then the next line yields $(\hat y, \hat z) = (1, 0)$ (unique).
                        The algorithm updates $\mathcal D = \{0, 1\}$. This is the same state as the beginning of case~1.
                  \item If $\tilde \xi = 1$, then the next line yields $(\hat y, \hat z) = (1, 0.5)$ (unique).
                        Since $\phi(\hat z, \hat \xi) = 0.5$, the algorithm updates $\mathcal D = \{0, 1\}$ and $\hat \lambda = 2$.
                        This state of the algorithm is also reached in case~1.
              \end{itemize}
    \end{enumerate}
    In all outcomes, the final $\hat \lambda \in \mathop{\arg\max}_{\lambda \in \mathbb{R}_{+}} \left\{ \max_{\xi \in \Xi} \mathcal{L}(x, \xi, \lambda) \right\}$ (see Figure~\ref{fig:counterexample_function}) and the algorithm terminates correctly with $0 = LB = UB = \max_{\xi \in \Xi} \mathcal{Q}(x, \xi)$.
\end{example}

Despite the positive result in the above example, it is nevertheless possible that the final value of $\lambda$ is suboptimal. %
A suboptimal choice of $\lambda$ may lead to an invalid upper bound and a premature termination of the algorithm.
As before, we can circumvent this \emph{ex post} by indirectly checking if a better choice of $\lambda$ can lead to higher upper bound.
In particular, after running the entire column-and-constraint generation algorithm to obtain the (candidate) optimal first-stage decisions $\bm x$, one can solve problem~\eqref{eq:worst_case_problem_discrete}, shown below.

\begin{equation}\label{eq:worst_case_problem_discrete}
    \begin{aligned}
        \mathop{\text{maximize}}_{\eta, \bm \xi, \bm \mu, \bm \rho} \;\;
         &
        \eta
        \\
        \text{subject to} \;\;
         & \eta \in \mathbb R, \;\; \bm \xi \in \Xi, \\
         &
        \left.
        \begin{aligned}
             &
            \bm{\mu}^{(k)} \in \mathbb{R}^m_{+}, \;\; \bm{\rho}^{(k)} \in \mathbb{R}^{n_p}_{+},
            \\
             &
            \eta \leq
            \bm{c}(\bm{\xi})^\top \bm{x}  +  \bm{d}_\mathrm{d}(\bm\xi)^\top \bm y_\mathrm{d}^{(k)} + \one^\top \bm \rho^{(k)},
            \\
             &
            \phantom{\eta \leq}
            + (\bm h_0 - \bm T\bm x - \bm W_\mathrm{d} \bm y_\mathrm{d}^{(k)} )^\top\bm\mu^{(k)},
            \\
             &
            \bm{W}_{\mathrm{c}}^\top \bm{\mu}^{(k)} \leq  \bm{d}_\mathrm{c}(\bm{\xi}),
            \\
             & \xi_j = 1 \implies \rho_j^{(k)} \leq \one^\top_j \bm{H}^\top \bm{\mu}^{(k)}, \;\; j \in [n_p], \\
             & \xi_j = 0 \implies \rho_j^{(k)} \leq 0, \;\; j \in [n_p],
        \end{aligned}\right\} \;\; k \in [|\mathcal D|].
    \end{aligned}
\end{equation}

Similar to Algorithm~\ref{algo:indicator:updated}, if the optimal value of the above problem happens to be strictly larger than the final estimate $UB$, then we can simply use its optimal solution to update $\lambda$ (see Theorem~\ref{theorem:optimal_multiplier_ex_post_general} below) and restart the column-and-constraint generation algorithm.
In doing so, we can retain the enumerated set $\mathcal D$ without re-initializing it to be empty.
As in the case of problem~\ref{eq:two_stage_ro_indicator}, we higlight that all previously generated constraints will continue to remain valid.
Similarly, if the loop terminating in line~\ref{algo:ccg:inner:optimality:necessary-condition} outputs an optimal multiplier, then problem~\eqref{eq:worst_case_problem_discrete} will be solved only at most once during the entire algorithm.

The following theorem is the counterpart of Theorem~\ref{theorem:optimal_multiplier_ex_post_indicator} for problem~\ref{eq:two_stage_ro_general}.%
\begin{theorem}\label{theorem:optimal_multiplier_ex_post_general}
    Suppose $\mathcal D \subseteq \mathcal Y_\mathrm{d}$ and $\bm x \in \mathcal X$ is any feasible first-stage decision in problem~\ref{eq:two_stage_ro_general}.
    Let $(\hat{\eta}, \hat{\bm\xi}, \hat{\bm\mu}, \hat{\bm\rho})$ denote an optimal solution of problem~\eqref{eq:worst_case_problem_discrete}.
    Then, $\bar \lambda = \max_{k \in [|\mathcal D|]}\left\{ \lVert \hat{\bm \rho}^{(k)} \rVert_\infty, \lVert \bm H^\top \hat{\bm \mu}^{(k)} \rVert_\infty \right\}$ is an optimal multiplier satisfying
    \begin{equation*}
        \sup_{\bm \xi \in \Xi} \inf_{\bm y_\mathrm{d} \in \mathcal D} \mathcal Q(\bm x, \bm \xi; \bm y_\mathrm{d})
        =
        \sup_{\bm \xi \in \Xi} \inf_{\bm y_\mathrm{d} \in \mathcal D} \mathcal L(\bm x, \bm \xi, \bar{\lambda}; \bm y_\mathrm{d}).
    \end{equation*}
\end{theorem}
\begin{proof}
    Define $\hat{\bm \beta}^{(k)} = \bar{\lambda} \hat{\bm\xi} - \hat{\bm \rho}^{(k)}$ for $k \in [|\mathcal D|]$.
    We claim that $(\bar{\lambda}, \hat{\eta}, \hat{\bm\xi}, \hat{\bm\mu}, \hat{\bm\beta})$ is feasible in problem~\eqref{eq:worst_case_problem_discrete_lagrangian_duality} with $(\tau, \lambda) = (1, \bar{\lambda})$.
    To see why, first note that if $\hat{\xi}_j = 1$, then $- \hat{\rho}^{(k)}_j \geq 0$ by definition of $\hat{\bm \rho}^{(k)}$ and if $\hat{\xi}_j = 1$, then $\bar{\lambda} - \hat{\rho}^{(k)}_j \geq 0$ by definition of $\bar{\lambda}$.
    The two cases together imply $\hat{\bm \beta}^{(k)} \geq \bm 0$.
    Next, note that the expression, $2 \bar{\lambda} \hat{\bm{\xi}} - \bm{H}^\top \hat{\bm{\mu}}^{(k)} - \hat{\bm{\beta}}^{(k)} = \hat{\bm{\rho}}^{(k)} - \bm{H}^\top \hat{\bm{\mu}}^{(k)}$, satisfies $\hat{\rho}^{(k)}_j - \one^\top_j \bm{H}^\top \hat{\bm{\mu}}^{(k)} \leq - \one^\top_j \bm{H}^\top \hat{\bm{\mu}}^{(k)}$ whenever $\hat{\xi}_j = 0$, and it satisfies $\hat{\rho}^{(k)}_j - \one^\top_j \bm{H}^\top \hat{\bm{\mu}}^{(k)} \leq 0$ whenever $\hat{\xi}_j = 1$.
    These two cases together with the definition of $\bar{\lambda}$ imply $2 \bar{\lambda} \hat{\bm{\xi}} - \bm{H}^\top \hat{\bm{\mu}}^{(k)} - \hat{\bm{\beta}}^{(k)} \leq \bar{\lambda} \one$.
    In summary, $(\bar{\lambda}, \hat{\eta}, \hat{\bm\xi}, \hat{\bm\mu}, \hat{\bm\beta})$ is feasible in problem~problem~\eqref{eq:worst_case_problem_discrete_lagrangian_duality} with $(\tau, \lambda) = (1, \bar{\lambda})$.
    The optimal value of the latter problem is precisely equal (by construction) to the right-hand side problem of the equation stated in this theorem.
    If $Z$ denotes the optimal value of problem~\eqref{eq:worst_case_problem_discrete}, then this implies
    \[
        Z \leq \sup_{\bm \xi \in \Xi} \inf_{\bm y_\mathrm{d} \in \mathcal D} \mathcal L(\bm x, \bm \xi, \bar{\lambda}; \bm y_\mathrm{d}).
    \]

    Define now $\mathcal T$ to be the (bilinear) optimization problem, which is identical to problem~\eqref{eq:worst_case_problem_discrete_lagrangian_duality} with $\tau = 1$ and the addition of $\lambda \geq 0$ as a decision variable.
    By construction, its optimal value $Z_{\mathcal T}$ satisfies the relation,
    \begin{equation*}
        \begin{aligned}
            Z
            \leq
            \sup_{\bm \xi \in \Xi} \inf_{\bm y_\mathrm{d} \in \mathcal D} \mathcal L(\bm x, \bm \xi, \bar{\lambda}; \bm y_\mathrm{d})
            \leq
            Z_{\mathcal T}
             & =
            \sup_{\lambda \geq 0} \sup_{\bm \xi \in \Xi} \inf_{\bm y_\mathrm{d} \in \mathcal D} \mathcal L(\bm x, \bm \xi, \lambda; \bm y_\mathrm{d}) \\
             & =
            \sup_{\bm \xi \in \Xi} \inf_{\bm y_\mathrm{d} \in \mathcal D} \mathcal Q(\bm x, \bm \xi; \bm y_\mathrm{d}),
        \end{aligned}
    \end{equation*}
    where the first inequality was established in the previous paragraph and the second equality follows by strong duality~\citep[Theorem~1]{subramanyam2022lagrangian}.
    Observe now that if we have $Z \geq Z_{\mathcal T}$, then the equation in the theorem follows.
    To see why $Z \geq Z_{\mathcal T}$, let $(\tilde{\lambda}, \tilde{\eta}, \tilde{\bm\xi}, \tilde{\bm\mu}, \tilde{\bm\beta})$ be a feasible solution in problem~$\mathcal T$ with objective value $\tilde{\eta}$. Then, an identical line of argument as before establishes that $(\tilde{\eta}, \tilde{\bm\xi}, \tilde{\bm\mu}, \tilde{\bm\rho})$ is feasible in problem~\eqref{eq:worst_case_problem_discrete} with the same objective value, where we define $\tilde{\bm \rho}^{(k)} = \bar{\lambda} \tilde{\bm\xi} - \tilde{\bm \beta}^{(k)}$ for all $k \in [|\mathcal D|]$. %
\end{proof}

\subsection{Impact on Computational Performance and Practical Considerations}\label{sec:computational}

Although the proposed modifications ensure correctness of the overall method and its finite termination to an optimal solution of the two-stage problem, they may be computationally intensive because of the presence of the indicator constraints in problems~\eqref{eq:worst_case_problem_indicator} and~\eqref{eq:worst_case_problem_discrete}.
Indeed, the main motivation of the original paper in developing the Lagrangian method was to circumvent the use of indicator constraints (and arbitrary upper bounds on dual variables), which were common in previous solution methods.
Indicator constraints are supported by very few solvers and their presence often significantly slows down the overall search process.
The lack of indicator constraints and arbitrary upper bounds directly contribute to the large computational speedups of the Lagrangian method over other solution methods.
Problems~\eqref{eq:worst_case_problem_indicator} and~\eqref{eq:worst_case_problem_discrete} should therefore be solved as few times as possible.

With a goal toward offering practical guidelines, we perform experiments across the entire set of 378 benchmark problems from~\citet{subramanyam2022lagrangian}.
This includes Benders decomposition for solving instances of~\ref{eq:two_stage_ro_indicator} and column-and-constraint generation for solving instances of~\ref{eq:two_stage_ro_general} and~\ref{eq:two_stage_ro_indicator}, with both continuous and mixed-integer decisions, across three problem classes.
The first two problem classes (network design and facility location) are instances of~\ref{eq:two_stage_ro_indicator} with continuous second-stage decisions, whereas the third (staff rostering) is an instance of~\ref{eq:two_stage_ro_general} with mixed-integer second-stage decisions.

None of the three problem classes satisfy any conditions that allow closed-form expressions of the optimal Lagrange multiplier.
The original paper explicitly acknowledges this for the first two problem classes but incorrectly assumes that the third class satisfies the sufficient conditions.
We therefore re-run the latter class of instances using the updated Algorithm~\ref{algo:ccg:inner}.
In doing so, we specify an initial multiplier value of $\lambda^0 = u(\bm x) - \ell(\bm x)$.
Interestingly, we find that this initial multiplier value already satisfies the necessary conditions in line~\ref{algo:ccg:inner:optimality:necessary-condition} of Algorithm~\ref{algo:ccg:inner} across all runs of all instances.
	
We implemented our experiments in Julia using JuMP as the modeling language and CPLEX~22.1.1 (with default options) as the solver. All runs were conducted on an Intel~Xeon~2.8 GHz computer, with a limit of 30~GB and eight cores per run. Since both the solver and hardware are different compared to the original paper, we re-run the original method so as to draw fair comparisons. For brevity, we only report results with the column-and-constraint generation algorithm as it was found to be computationally superior to Benders decomposition in the experiments of the original paper.
	
Table~\ref{table:results} summarizes the computational overhead of the proposed algorithmic modifications on all three benchmark problem classes. For ease of presentation, we only report statistics across those instances for which the original method terminated within the time limit (one~hour for the network design problem and two~hours for the facility location and staff rostering problem). The column `Opt' reports the number of instances for which the final solution obtained by that method was in fact, provably optimal. In the case of solutions obtained using the original method, this is checked \emph{ex post} by solving the indicator-constrained problems~\eqref{eq:worst_case_problem_indicator} and~\eqref{eq:worst_case_problem_discrete}.
For fair comparison, we do not include their solution times in the corresponding column `$t$ (s)'.
In the case of the proposed modification, we not only include this time in the corresponding column `$t$ (s)', but we also include any additional iterations and runtimes that may have been incurred by restarting the original method with the updated value of $\lambda$. The columns `\#It.' and `$n_\mathrm{restarts}$' report the corresponding average number of iterations\footnote{For the staff rostering problem, the column `\#It.' reports the total number of iterations in the inner-level column-and-constraint generation algorithm (summed across all outer-level iterations), to be consistent with the original paper.} and total number of restarts.
	
	\begin{table}[!htbp]
		\centering
		\caption{Computational performance of the proposed algorithmic modifications.}
		\label{table:results}
		\begin{tabularx}{\textwidth}{r*{6}{R}r}
			\toprule
			& \multicolumn{3}{c}{Original method} & \multicolumn{3}{c}{Proposed modification} & \\
			\cmidrule(r){2-4}\cmidrule(l){5-7}
			Instance & Opt$^*$ & \#It.$^\dagger$ & $t$ (s)$^\dagger$ &Opt & \#It. & $t$ (s) &  $n_\mathrm{restarts}$ \\
			\midrule
			& & & \multicolumn{2}{c}{Network design} & & & \\
			dfn-bwin & 16/16 &    67.4 &    33.6 & 16/16 &    67.4 &    41.2 &  0 \\
			dfn-gwin & 10/10 &    66.8 &    11.7 & 10/10 &    66.8 &    14.5 &  0 \\
			di-yuan & 11/12 &    88.7 &    13.7 & 12/12 &    88.8 &    13.4 &  1 \\
			\cmidrule{2-8}
			& 37/38 &    73.9 &    21.6 & 38/38 &    74.0 &    25.4 &  1 \\
			\midrule
			& & & \multicolumn{2}{c}{Facility location} & & & \\
			$k=1$ &  35/35   &     6.6 &    10.1 &  35/35   &     6.6 &     9.7 &  0 \\
			$k=2$ &  35/35   &    16.3 &   331.3 &  35/35   &    16.3 &   335.0 &  0 \\
			$k=3$ &  26/26   &    29.0 &   625.5 &  26/26   &    29.0 &   608.1 &  0 \\
			$k=4$ &  18/18   &    26.2 &   196.1 &  18/18   &    26.2 &   190.0 &  0 \\
			\cmidrule{2-8}
			& 114/114  &    17.8 &   278.5 & 114/114  &    17.8 &   274.5 &  0 \\
			\midrule
			$(T, k)$ & & & \multicolumn{2}{c}{Staff rostering} & & & \\
			(12, 3)  & 10/10  &    55.2 &    20.2 & 10/10  &    55.2 &    20.0 &  0 \\
			(12, 6)  & 10/10  &    87.1 &   182.7 & 10/10  &    87.1 &   196.6 &  0 \\
			(12, 9)  & 10/10  &    72.3 &    18.3 & 10/10  &    72.3 &    21.6 &  0 \\
			(24, 3)  &  3/3  &   567.0 &  2362.0 &  3/3  &   567.0 &  2816.0 &  0 \\
			(24, 6)  &  1/1  &   970.0 &  3050.4 &  1/1  &   970.0 &  3145.7 &  0 \\
			\midrule
			 & 34/34  &   141.7 &   363.2 & 34/34  &   141.7 &   411.1 &  0 \\
			\bottomrule
			\multicolumn{8}{l}{\footnotesize$^\dagger$ Differences compared to~\cite{subramanyam2022lagrangian} are due to solver and hardware alone.}\\
			\multicolumn{8}{l}{\footnotesize$^*$ Exactness of the original method is decided \emph{ex post} using the proposed modification.}\\
			\multicolumn{8}{l}{\footnotesize\phantom{$^*$} The corresponding runtimes are not included in the adjoining column `$t$ (s)'.}
		\end{tabularx}
	\end{table}

Table~\ref{table:results} shows that the computational overhead of the proposed algorithmic modifications is minor. The number of iterations is identical to the original method in all but one instance. In other words, the \emph{ex post} solution of problems~\eqref{eq:worst_case_problem_indicator} and~\eqref{eq:worst_case_problem_discrete} revealed that their optimal values are equal to those computed using the original methods in all but one benchmark instance. In fact, although not shown in Table~\ref{table:results}, we verified that the final upper bounds computed by the original method were provably rigorous, even in those instances that did not terminate within the time limit.
In other words, the final Lagrange multipliers computed in~\citet{subramanyam2022lagrangian} are provably optimal across 377 out of 378 benchmark instances.
Moreover, the objective values of the corresponding final solutions rigorously upper bound (or equal, if terminated within the time limit) the optimal value of the original two-stage problem.
Moreover, in the only case where this was not true\footnote{Network design instance `di-yuan' with budget parameter $k=3$.}, the estimated upper bound was less than its true optimal value by less than 0.3\%, and the algorithm had to be restarted only once in this case.

Table~\ref{table:results} shows that the solution times of the original method remain mostly unchanged even with the inclusion of the proposed modifications. This is not surpring given that the indicator constrained optimization problems \eqref{eq:worst_case_problem_indicator} and \eqref{eq:worst_case_problem_discrete} are solved only once in most instances. Also, although not shown in the table, the
solution times of the indicator constrained optimization problems \eqref{eq:worst_case_problem_indicator} and \eqref{eq:worst_case_problem_discrete} are in fact slower than their counterparts with fixed $\lambda$.
Compared to the latter, their solution times are slower by a (geometric average) factor of 11.9 (network design), 1.3 (facility location), and 2.8 (staff rostering).
However, because the indicator constrained problems end up being solved only once during the entire algorithm, this slower solution time has a negligible effect on the overall computational times, which remain similar to those originally reported in~\citet{subramanyam2022lagrangian}.

We offer the following conclusions based on our observations. First, the necessary conditions that are enforced within the algorithms of the original paper appear to be almost sufficient in experiments, hinting at opportunities to generalize the conditions in Theorems~\ref{theorem:optimal_multiplier_corrected} and~\ref{theorem:optimal_multiplier_indicator}. Moreover, the computational speedups offered by the original algorithms over traditional methods are only possible because of their use of Lagrangian functions (with fixed $\lambda$) as proxies for the second-stage value function. %
We showed how to check if these proxies are rigorous by solving alternate but computationally difficult optimization problems. We therefore suggest to solve these only at the end to assess the potential suboptimality of the final first-stage decisions. If the optimality gap is not satisfactorily small, then one can obtain better solutions by warm-starting the original algorithms in the manner we showed in this paper.